\documentclass[a4paper, 12pt]{amsart}
\usepackage{amssymb, amsmath, amsthm, enumerate, color}
\usepackage{mathtools, amstext, tikz, mathrsfs, mathabx, pgfplots, stmaryrd}
\usepackage[colorlinks=true, breaklinks=true, linkcolor=black, citecolor=blue, urlcolor=red]{hyperref} 
\usepackage[english]{babel}
\usepackage{forest, adjustbox}

\numberwithin{equation}{section}
\newtheorem{letterthm}{Theorem}

\newtheorem{lettercor}[letterthm]{Corollary}

\newtheorem{letterdefinition}[letterthm]{Definition}

\newtheorem{theorem}{Theorem}[section]
\newtheorem{lemma}[theorem]{Lemma}
\newtheorem{corollary}[theorem]{Corollary}
\newtheorem{proposition}[theorem]{Proposition}
\newtheorem{question}[theorem]{Question}

\newtheorem{definition}[theorem]{Definition}

\newtheorem{remark}[theorem]{Remark}
\newtheorem{example}[theorem]{Example}

\newcommand{\act}{\curvearrowright}

\newcommand{\cC}{\mathcal C}

\newcommand{\C}{\mathbf C}

\newcommand{\fF}{\mathfrak F}
\DeclareMathOperator{\Fix}{Fix}

\newcommand{\fH}{\mathfrak H}

\newcommand{\scrH}{\mathscr H}

\DeclareMathOperator{\Homeo}{Homeo}
\newcommand{\cI}{\mathcal I}

\DeclareMathOperator{\id}{id}
\DeclareMathOperator{\Ind}{Ind}

\newcommand{\cJ}{\mathcal J}

\newcommand{\fK}{\mathfrak K}
\newcommand{\scrK}{\mathscr K}

\newcommand{\la}{\lambda}
\DeclareMathOperator{\Leaf}{Leaf}

\newcommand{\N}{\mathbf{N}}

\newcommand{\NInd}{\text{Ind-mixing}}

\newcommand{\ot}{\otimes}

\newcommand{\cP}{\mathcal P}

\newcommand{\R}{\mathbf{R}}

\DeclareMathOperator{\Ran}{Ran}

\DeclareMathOperator{\Root}{Root}
\newcommand{\bS}{\mathbf S^1}
\DeclareMathOperator{\Stab}{Stab}

\DeclareMathOperator{\supp}{supp}
\newcommand{\fT}{\mathfrak T}
\DeclareMathOperator{\tar}{tar}
\newcommand{\ti}{\tilde}

\newcommand{\cU}{\mathcal U}

\newcommand{\fU}{\mathfrak U}

\DeclareMathOperator{\Ver}{Ver}

\newcommand{\fX}{\mathfrak X}

\newcommand{\Y}{\wedge}
\newcommand{\Z}{\mathbf{Z}}

\DeclarePairedDelimiterX{\norm}[1]{\lVert}{\rVert}{#1}
\DeclareMathAlphabet\urwscr{U}{urwchancal}{m}{n}%

\begin{document}

	\title[Jones' representations]{Jones' representations of R.~Thompson's groups not induced by finite-dimensional ones}
	\thanks{
		AB is supported by the Australian Research Council Grant DP200100067.}
	\author{Arnaud Brothier and Dilshan Wijesena}
	\address{Arnaud Brothier, Dilshan Wijesena\\ School of Mathematics and Statistics, University of New South Wales, Sydney NSW 2052, Australia}
	\email{arnaud.brothier@gmail.com\endgraf
		\url{https://sites.google.com/site/arnaudbrothier/}}
		\email{dilshan.wijesena@hotmail.com}
	\maketitle
	
	\begin{abstract}
		Given any linear isometry from a Hilbert space to its square one can explicitly construct a so-called Pythagorean unitary representation of Richard Thompson's group $F$.
		We introduce a condition on the isometry implying that the associated representation does not contain any induced representations by finite-dimensional ones. 
		This provides the first result of this kind.
		We illustrate this theorem via a family of representations parametrised by the real 3-sphere for which all of them have this property except two sub-circles.
	\end{abstract}
	
	%\keywords{{\bf Keywords:} Thompson's groups, unitary representations, Jones' representations, mixing, fraction groups, Pythagorean C*-algebras, }

\section*{Introduction}
Vaughan Jones introduced a powerful tool for constructing actions of groups known as {\it Jones' technology} \cite{Jones17, Jo18}.
In particular, any linear isometry $\fH\to\fH\oplus\fH$ with $\fH$ a Hilbert space provides a so-called {\it Pythagorean} (unitary) representation $(\sigma,\scrH)$ of Richard Thompson's groups $F,T,$ and $V$ \cite{BJ19}.
The strength of this construction resides in obtaining unitary representations of the complicated groups $F,T,$ and $V$ using elementary initial data (for instance a linear isometry $\C^n\to\C^{2n}$). Moreover, these representations are very explicit carrying obvious algorithms for computing matrix coefficients. Finally, we can surprisingly derive representations of the Cuntz algebra making this technology also appealing from an operator algebraic point of view.

This article is a first of a series that will develop techniques and tools to study these Thompson's groups representations in view of decomposing them into irreducible components, deciding which one are pairwise isomorphic, and how much they differ from obvious ones (like monomial representations associated to subgroups of $F$).
In this article we introduce a simple criteria named {\it diffuse} on the linear isometry $\fH\to\fH\oplus\fH$ assuring that the associated unitary representation $(\sigma,\scrH)$ does not contain any representation induced by finite-dimensional ones (we refer to this latter property as \textit{\NInd}). 

{\bf Background.}
{\it Richard Thompson's groups $F,T,$ and $V$.}
Richard Thompson's groups introduced three groups $F\subset T\subset V$ in unpublished notes during the 1960's, see  \cite{CFP}. 
They are countable discrete groups acting by homeomorphisms on the unit interval, the unit circle, and the Cantor set, respectively. 
They appear in various area of mathematics such as topology, logic, dynamics and more recently in the reconstruction program of conformal field theories of Jones, see \cite{Jones17, BSurvey} for details on Jones' connection.
They are famous for satisfying rare properties of groups and following unexpected behaviours. 
One of the most celebrated open question regarding them is to know if $F$ is amenable or not. Note that much weaker properties than amenability are not known to hold including exactness, sophicity, and Cowling-Haagerup's weak amenability. 
Deeply understanding groups is usually done by studying actions of them. 
Jones introduced a technology to construct such.

{\it Jones' technology.}
In the 2010's Jones came across the Thompson groups while he was aiming to construct conformal field theories from subfactors \cite{Jones17, BSurvey}.
This connection came from the similarity of Kenneth Brown's diagrammatic description of elements of $F,T,V$ via rooted finite binary trees and the string diagrams of Jones (forming a planar algebra) used to describe the standard invariant of a subfactor \cite{brown1987,Jones-PA}.
From that discovery he defined a novel technique for constructing actions of the Thompson groups (and many other such groups) \cite{Jones17,Jo18}. 
This has already lead to numerous applications in mathematical-physics, group theory, knot theory, noncommutative probability and so on \cite{Jones18-Hamiltonian, Brothier-Stottmeister19, brothier2019haagerupT, Brothier19WP,Jones19-survey, Kostler-Krishnan-Wills20}. 

{\it Pythagorean's representations.}
Among other, this novel technology of Jones permits to construct a unitary representation of $F,T,$ and $V$ from any linear isometry $\fH\to \fH\oplus \fH$ where $\fH$ is a Hilbert space \cite{BJ19}.
These representations are named {\it Pythagorean} for the following reason: if we write the linear isometry as $\xi\mapsto (A\xi,B\xi)$ for some bounded linear operators $A,B\in B(\fH)$, then this pair of operators must satisfy the relation reminiscent of the one of Pythagoras:
\begin{equation}\label{eq:PR}A^*A + B^*B=\id_{\fH}.\end{equation}
The associated unitary representation $(\sigma,\scrH)$ of $F$ is constructed via an inductive limit of direct sum powers of $\fH$. It heavily depends on the choice of the so-called {\it Pythagorean pair} $(A,B)$. 
The limit Hilbert space $\scrH$ can be thought as the set of classes of (finite rooted binary) trees with their (ordered) leaves indexed by elements of $\fH$. Moving inside an equivalence class corresponds in growing or reducing trees and applying the operators $A,B$ to the decorations of the leaves, see Subsection \ref{sec:def-pyth} for details. 

{\it Connection with C*-algebras.} 
This extends a previous construction of Nekrashevych where moreover $A^*,B^*$ were required to be themselves isometries (forcing $\fH$ to be infinite-dimensional) and using the Cuntz algebra $\mathcal O$ \cite{nekrashevych2004cuntz,cuntz1977simple}.
One of the main power of Pythagorean representations is to be able to construct interesting representation for $F,T,V$ but starting from {\it finite-dimensional} operators $A,B$.
%C*-algebra
On the C*-algebraic side one can see that a Pythagorean representation is nothing else than a representation of the universal C*-algebra $P$ defined by two operators satisfying \eqref{eq:PR}. This is one of the few known non-nuclear C*-algebra having the lifting property as proved by Courtney \cite{courtney2021universal}. 
There is an obvious surjective map from $P$ to $\mathcal O$ and thus any representation of $\mathcal O$ defines one of $P$.
Surprisingly, any representation of $P$ produces a representation of $\mathcal O$ giving a new way to construct such using again only finite-dimensional initial data, see \cite[Proposition 7.1]{BJ19}. 

{\bf Content and main results.}
In this article we consider a Pythagorean representation $(\sigma,\scrH)$ of Thompson group's $F$ constructed from a Pythagorean pair $(A,B)$ acting on some $\fH$. 
We focus on actions of $F$ even if the representation $\sigma$ canonically extends to $T$ and $V$, making our analysis still relevant for these larger groups. 
Our approach consists in obtaining properties of representation of $F$ (thus acting on a {\it large} inductive limit Hilbert space $\scrH$) from the study of the two operators $A,B$ acting on a comparably {\it smaller} Hilbert space $\fH$.
In this article we compare Pythagorean representations with monomial ones (e.g.~quasi-regular) or more generally representations induced by finite-dimensional ones (and thus easily constructible without using Jones' technology). 
This leads to the following natural notion for unitary representations of discrete groups.
\begin{letterdefinition}\label{letter-def:NInd}
	Let $\sigma:G\act \scrH$ be a unitary representation of a discrete group. 
	We say that $\sigma$ is {\it \NInd} if given any non-trivial subgroup $K\subset G$ and any finite-dimensional non-zero unitary representation $\theta:K\act\fK$ we have that the induced representation $\Ind_K^G\theta$ is not contained in $\sigma$.
\end{letterdefinition}
Note that if $G$ is torsion-free, then mixing implies \NInd; and for all groups \NInd\ implies weak mixing, see Section \ref{sec:prop-rep} for definitions. However, none of the reverse implications hold in general.
We introduce the following key notion for Pythagorean pairs.
\begin{letterdefinition}\label{letter-def:diffuse}
	A Pythagorean pair $(A,B)$ acting on $\fH$ is called {\it diffuse} if given any increasing sequence of words $p_n$ in $A,B$ and any vector $\xi\in\fH$ we have that 
	$$\lim_{n\to\infty} p_n \xi =0.$$
	In that case, we say that the associate Pythagorean representation $(\sigma,\scrH)$ is diffuse.
\end{letterdefinition}

Here is our main result linking the notion of diffuse Pythagorean pairs with \NInd\ representations.
\begin{letterthm}\label{letter-theo}
	A diffuse Pythagorean representation is \NInd. 
\end{letterthm}

Note that all Pythagorean representations are not mixing, see Remark \ref{obs:NInd-mixing}.
We obtain the first known family of representations of $F$ of this kind (being \NInd\ but not mixing).
A related result was proved by Garncarek when he considered a one-parameter deformation of the Koopman representation of $F\act [0,1]$ \cite{garncarek2012analogs}. He proved, among other, that these representations do not contain induced representations of the form $\Ind_{F_p}^F \theta$ where $F_p$ is a parabolic subgroup of $F$ and $\theta:F_p\act \fK$ is finite-dimensional. 

We illustrate our theorem by considering all Pythagorean representations obtained from all linear isometries $\C\to\C^2$.
They are parametrised by the unit vectors $(a,b)$ of $\C^2$ which is the real 3-sphere $S^3$ (hence our operators $A,B$ are complex scalar multiplications by $a,b$ acting on $\C$).
The diffuse representations are exactly those where both $a$ and $b$ are non-zero, that is, $S^3$ minus the union of the two circles
$$C_1:=\{(a,0)\in\C^2:\ |a|=1\} \text{ and } C_2:=\{(0,b)\in\C^2:\ |b|=1\}.$$
We directly deduce the following.
\begin{lettercor}\label{letter-cor:dim-one}
	For all $(a,b)\in S^3\setminus(C_1\cup C_2)$ we have that the associated representation $\sigma_{a,b}:F\act\scrH$ is \NInd.
\end{lettercor}
We observe that the class of representations considered by Garncarek corresponds exactly to the sub-circle $C_3\subset S^3$ of all $(\omega/\sqrt 2,\omega/\sqrt 2)$ with $\omega$ a complex number of modulus one. 
We will show in a future work that the diffuse representations of this corollary are all irreducible and pairwise non-isomorphic.
This extends and generalises Garncarek's pioneering results using novel techniques that deeply rely on the Pythagorean structure of the representations.
We will later provide a much wider and richer class of concrete irreducible \NInd\ examples by considering $A,B$ acting on higher dimensional Hilbert spaces.

By adapting the proof of Theorem \ref{letter-theo} we obtain the following criteria on $(A,B)$ assuring that $\sigma:F\act\scrH$ is weakly mixing.

\begin{lettercor}\label{letter-cor}
	Consider a Pythagorean pair $(A,B)$ acting on $\fH$ with associated representation $(\sigma,\scrH)$.
	If $\lim_{n}A^n\xi=\lim_n B^n\xi=0$ for all $\xi\in\fH$, then $\sigma$ is weakly mixing (i.e.~does not contain any finite-dimensional subrepresentation).
\end{lettercor}

In a future work (which relies on general decompositions of Pythagorean representations) we will show that both assumptions of Theorem \ref{letter-theo} and Corollary \ref{letter-cor} are in fact necessary \cite{Brothier-Wijesena-2}. Hence, we provided characterisations of \NInd\ and weak mixing Pythagorean representations solely in terms of the pair $(A,B)$. 

{\bf Plan of the article.}
In Section \ref{sec:preliminaries} we provide a detailed preliminary section on unitary representations, strong operator topology, Thompson's groups, and Pythagorean representations. 
This allows us to introduce terminology, notations, and classical results needed for our study.

In Section \ref{pythag rep section} we introduce a class of partial isometries $(\tau_\nu)_\nu$ acting on $\scrH$ that is the carrier Hilbert space of a Pythagorean representation $\sigma$. 
These partial isometries are indexed by finite binary sequences $\nu$ (or equivalently vertices of the rooted infinite binary tree).
They permit to decompose vectors $\xi\in\scrH$ in a diagrammatic fashion and moreover to describe the action $\sigma:F\act \scrH$.
In particular, if $g\in F$ is described by a pair of trees with leaves indexed by $\nu_i$ and $\mu_i$, then $\sigma(g)$ is equal to the finite sum of operators $\sum_i\tau_{\nu_i}^*\circ \tau_{\mu_i}$. 
This decomposition of both the vectors and the actions of $F$ using the $\tau_\nu$ plays a major role in our study. 
We establish a number of useful observations on these partial isometries.
Moreover, we obtain a corollary of von Neumann's ergodic theorem which relates projections of $\scrH$ and supports of Thompson's group elements.
This corollary will be key for proving Theorem \ref{letter-theo}.

In Section \ref{induced chapter} we prove our main result, Theorem \ref{letter-theo}, via the following strategy.
We fix a diffuse Pythagorean pair $(A,B)$ acting $\fH$ with associated representation $\sigma:F\act\scrH$.
For the sake of contradiction we assume that $\Ind_H^F\theta\subset\sigma$ where $\theta:H\act \fK$ is a non-trivial finite-dimensional representation of a subgroup $H\subset F$.
We introduce the set $$\mathscr P:=\{g\in F:\ g^n\notin H \text{ for all } n\geq 1\}.$$
Using projections defined by vertices and the von Neumann ergodic theorem we show that there exists a nonempty subset $Q$ of the Cantor space (i.e.~the set of all infinite binary sequences) satisfying that if $Q\not\subset\supp(g)$, then $g^n\in H$ for infinitely many $n\geq 1.$
We fix $u\in Q$.
Using tree-diagrams and the assumption we construct a subset of elements $\tilde H\subset F$ such that:
\begin{enumerate}
\item they all act trivially around $u$ and 
\item given any finite subset $K\subset \scrH$ we can extract a sequence $\{g_j\}_j\subset \tilde H$ with the property $\lim_j\langle \sigma(g_j) \xi,\xi\rangle=0$ for all $\xi\in K$.
\end{enumerate}
Since any $g\in \tilde H$ acts trivially around $u$ we have that $u\notin\supp(g)$ and thus $Q\not\subset \supp(g)$. Therefore, $g^n$ is not in $H$ for infinitely many $n\geq 1$.
This allows us to further assume that the sequence $\{g_j\}_j$ of above is contained in $H$.
We deduce that the restriction $\sigma\restriction_H$ is weakly mixing. 
This produces a contradiction since $\theta\subset \sigma\restriction_H$ is a non-trivial finite-dimensional subrepresentation of a weakly mixing one.
By slightly adjusting this proof we deduce Corollary \ref{letter-cor}.

We end the paper with Section \ref{sec:example} in which we describe the class of Pythagorean representations obtained from $(A,B)$ acting on $\fH=\C$.

%%%%%%%%%%%%%%%%%END INTRODUCTION%%%%%%%%%%%%%%%%%%%%%%%%%%%%%%%
%%%%%%%%%%%%%%%%%%%%%%%%%%%%%%%%%%%%%%%%%%%%%%%%%%%%%%%%%%%

\section{Preliminaries}\label{sec:preliminaries}

We recall definitions and facts concerning unitary representations of groups, Richard Thompson's groups $F,T,V$, and a particular case of Jones' construction. 
All of the content of the preliminaries can be found elsewhere and thus we will be rather brief. 
We take the occasion of this section to introduce terminology, notations, and fixing conventions.
We recommend the reader to consult \cite{BHV, kerr2016ergodic, BH} for materials concerning unitary representations, \cite{CFP} for Thompson's groups, \cite{Jo18,brothier2019haagerupT, BSurvey} for the general theory of Jones' technology, and \cite{BJ19} for the particular case of Pythagorean representations that we will be exclusively using in this article.

{\bf Convention.} 
All along the article we take the convention that all groups are equipped with the discrete topology, all Hilbert spaces are over the complex field $\C$, and are linear in the first variable. Moreover, all representation are unitary.

\subsection{Unitary representations of discrete groups}

\subsubsection{Unitary representations}
A \textit{unitary representation} $\sigma:G\to \cU(\fH_\sigma)$ of a group $G$ is a group morphism from $G$ to the unitary group $\cU(\fH_\sigma)$ of a Hilbert space $\fH_\sigma$. 
We will be exclusively considering \textit{unitary} representations and thus may drop the word ``unitary''. 
A representation of a group $G$ will be denoted $(\sigma,\fH)$, $\sigma:G\to\cU(\fH)$, or $G\act \fH$.
As usual we write $\sigma\restriction_H$ for the {\it restriction} of $\sigma$ to a subgroup $H\subset G$.

\subsubsection{Induced representations}
If $G$ is a group, $H\subset G$ a subgroup, and $\sigma:H\to\cU(\fK)$ a representation, then we define the {\it induced} representation of $\sigma$ to be
\begin{align*}
&\Ind_H^G\sigma:  G \to \cU(\ell^2(T,\fK))\\
& \Ind_H^G\sigma(g)f(t) = \sigma(\beta(g^{-1},t)^{-1})(f(g^{-1} \cdot t),
\end{align*}
where $g\in G, t\in T, f\in \ell^2(T,\fK)$, $T$ is a fixed set of representatives of $G/H$, so that $(g\cdot t,\beta(g,t))$ is the unique pair of $T\times H$ satisfying $gt=(g\cdot t) \beta(g,t).$

When $\dim(\fK)=1$, then we say that $\Ind_H^G\sigma$ is {\it monomial}.
Note that when $\sigma=1_H$ is the trivial representation, then $\Ind_H^G1_H$ is isomorphic to the {\it quasi-regular} representation $\la_{G/H}:G\act \ell^2(G/H)$.

\subsubsection{Properties of representations}\label{sec:prop-rep}

Consider a representation $\sigma:G\to\cU(\fH_\sigma)$ and recall that a \textit{subrepresentation} of $\sigma$ is a (topologically) closed vector subspace $\mathfrak K\subset \fH_\sigma$ that is closed under the action of $G$.
The representation $\sigma$ is called:
\begin{enumerate}
	\item \textit{irreducible} if it does not admit any proper non-zero subrepresentation (and is called \textit{reducible} otherwise);
	\item \textit{mixing} if $G$ is infinite and $\lim_{g\to\infty} |\langle \sigma_g\xi,\xi\rangle|=0$ for all $\xi\in\fH_\sigma$;
	\item \textit{weakly mixing} if it does not admit any finite-dimensional non-zero subrepresentation;
\end{enumerate}

Here is a well-known reformulation of weakly mixing (which is often taken as the definition), see \cite[Theorem 2.23]{kerr2016ergodic} for a proof and details. 

\begin{proposition}\label{prop:unitary-rep}
	A representation $\sigma$ is weakly mixing if and only if for every finite subset $K \subset \fH_\sigma$ and $\epsilon > 0$, there exists $g \in G$ such that $|\langle \sigma_g\xi,\xi\rangle | < \epsilon$ for all $\xi \in K$.
\end{proposition}

From the definition it is clear that for infinite-dimensional representations irreducible implies weakly mixing. 
Proposition \ref{prop:unitary-rep} shows that mixing implies weak mixing.
Although, an irreducible representation is not necessarily mixing. 
Indeed, take any infinite self-commensurator subgroup $H\subset G$ and note that the quasi-regular representation $G\act \ell^2(G/H)$ is irreducible but not mixing from the Mackey-Schoda criterion \cite{Mack51} (see also \cite{Burger-Harpe97}). Take for instance $G=F$ and $H=F_{1/2}$ to be the parabolic subgroup of $g\in F$ fixing the dyadic rational $1/2$.

\subsection{Thompson's groups $F,T,$ and $V$}

In unpublished notes of 1960's Richard Thompson defined three groups $F,T,V$ satisfying $F\subset T\subset V$.
We will be focused on $F$ in this article but all the representations that we will define for $F$ extends to $V$ and thus this study stays relevant for all three Thompson's groups.

\subsubsection{Classical definitions as homeomorphism groups}

{\bf Description of $F$ as actions on $[0,1]$:}
Thompson's group $F$ is the group of piece-wise linear homeomorphisms of the unit interval $[0,1]$ having finitely many breakpoints all occurring at dyadic rationals $\Z[1/2]$ and having all slopes powers of $2$.
We will often identify $F$ as a subgroup of the homeomorphism group $\Homeo(0,1)$ of $[0,1].$

\subsubsection{Definition of $F$ involving partitions}
{\bf Standard dyadic interval.}
A common way to understand $F$ is to consider certain partitions of $[0,1]$.
A \textit{standard dyadic interval} (in short sdi) is an interval of the form $I=[\frac{a}{2^b}, \frac{a+1}{2^b}]$ with $a,b\in \N$ which is contained in $[0,1]$.
For technical reasons we may consider the half-open interval $\dot{I}:=[\frac{a}{2^b}, \frac{a+1}{2^b})$ and may identify $I$ with $\dot I$.

{\bf Standard dyadic partition.}
A \textit{standard dyadic partition} (in short sdp) is a finite list $\cI:=(I_1,\cdots,I_n)$ of sdi so that $\sup(I_j)=\inf(I_{j+1})$ for $1\leq j\leq n-1$ and so that $\inf(I_1)=0, \sup(I_n)=1$, i.e.~an ordered finite partition of $[0,1]$ into sdi up to removing certain endpoints of the sdi. 
Note that $\{\dot I_1,\cdots,\dot I_n\}$ is a partition of $[0,1)$ in the usual sense.

{\bf A partial order.}
If $\cI,\cJ$ are sdp and every sdi of $\cI$ is equal to the union of some sdi of $\cJ$, then we say that $\cI$ is a \textit{refinement} of $\cJ$ and write $\cI\leq \cJ$.
The set of all sdp equipped with $\leq$ is a directed partially ordered set (in short a directed poset). 

{\bf Description of $F$ using partitions and definitions of $T$ and $V$:}
Thompson's group $F$ is the group of maps $g:[0,1]\to[0,1]$ satisfying that there exists two sdp $(I_1,\cdots,I_n)$ and $(J_1,\cdots,J_n)$ for some common $n\geq 1$ so that $g$ maps bijectively in the unique increasing and affine way $I_k$ onto $J_k$ for each $1\leq k\leq n$.
To define $V$ we add in the data a permutation $\sigma$ of $\{1,\cdots,n\}$ so that $g$ maps $\dot I_k$ onto $\dot J_{\sigma(k)}$. The group $T$ is obtained by requiring that $\sigma$ is cyclic.
Note that elements of $T$ acts continuously on the torus $\R/\Z$ while elements of $V$ have finitely many discontinuous points (all at dyadic rationals).

\subsubsection{Description of $F$ acting on the Cantor space} \label{action cantor subsection}

\textbf{Cantor space.}
Recall, the Cantor space $\mathcal C:=\{0,1\}^{\N^*}$ is the space equal to all infinite sequences in $0, 1$ which is equipped with the product topology generated by the sets $I_w:= \{ w\cdot p: p\in \cC\}$ for all finite words $w$ in $0,1$. Note the sets $I_w$ are in fact open and closed since $\mathcal C\setminus I_w$ can be written as a finite union of $I_\nu$. The Cantor space is associated to the unit interval by the following map
\[S : \cC \rightarrow [0,1],\ (x_n : n \geq 1) \mapsto \sum_{n \in \N^*} 2^{-n}x_n\]
which is surjective and continuous. Each dyadic rational in $(0,1)$ has two preimages under $S$ while all other points have an unique preimage. In particular, for each finite word $w$, the binary sequences $(w0111\dots)$, $(w1000\dots)$ are mapped to the same dyadic rational. Additionally, eventually constant sequences are mapped to the dyadic rational, eventually periodic sequences are mapped to the rationals while not eventually periodic sequences are mapped to the irrationals. Furthermore, $S$ maps the sets $I_w$ to sdi's up to removing certain endpoints. This yields a bijection between $\{I_w : w \textrm{ is a finite word}\}$ and the set of sdi's. 

{\bf Notation.} For a sdi $I$, denote $m_I$ to be the finite word given by the above bijection. 

A desirable property of the above identification is that now each sdp of $[0,1]$ gives rise to a ``true'' partition of $\cC$ without needing to remove certain endpoints and removes the ambiguity of which sdi a dyadic rational in $(0,1)$ belongs to. 
This is a consequence of $\cC$ containing two copies of the dyadic rationals in $(0,1)$. 
For this reason in this paper we will often work on the Cantor space rather than the unit interval. Further, from hereon, we shall identify sdi's with the subsets $I_w$ of the Cantor space and similarly for sdp's.

\textbf{Description of $F$ as actions on the Cantor space.} Using the above correspondence and the definition of $F$ involving partitions we obtain an action of $F$ on the Cantor space in the obvious way. Indeed,
$F$ is the group of homeomorphism on $\cC$ satisfying there exists two ordered sdp's $(I_1, \dots, I_n)$ and $(J_1, \dots, J_n)$ for some common $n \geq 1$ such that $I_k$ is ``linearly'' mapped onto $J_k$ in the following way:
\begin{equation} \label{cantor action equation}
	m_{I_k} \cdot w \mapsto m_{J_k} \cdot w, \quad 1 \leq k \leq n, w \in \{0,1\}^{\N^*}.
\end{equation}
Similarly, we can obtain actions of $T,V$ as homeomorphisms on the Cantor space. It is clear that the map $S$ is $F$-equivariant with respect to the above action.

{\bf Support.}
The \textit{support} of an element $g \in F$ is the closure of all sequences in $\cC$ which are not fixed by $g$. That is, 
$$\supp(g) := \overline{\{p \in \cC : g(p) \neq p\}}.$$ From the description of $F$ acting on the Cantor space, it is clear that the support is always a finite disjoint union of sdi and is thus both open and closed.

\subsubsection{Diagrammatic description of elements of Thompson's groups} \label{diag desc subsection}

{\bf Infinite binary rooted tree.}
Consider the infinite binary rooted tree $t_\infty$ that we identify with a geometric realisation of it in the plane where the root is on top, the root has two neighbours placed at the bottom left and right of it, every other vertex has three neighbours: one above it, one on the bottom left, and one at the bottom right; the latter two being called \textit{immediate children} of the vertex.
We equip the tree with the orientation satisfying that each oriented edge is going from top to bottom. 
We call \textit{left and right edges} the oriented edges going in the bottom left and right directions, respectively.
Moreover, a pair of edges with a common source (going down) is called a \textit{caret}. We denote it with the symbol $\Y$.

{\bf Decoration of vertices with binary sequences.}
We will often identify vertices of $t_\infty$ with finite binary sequences (also referred to as finite words). 
Write $\Ver$ for the vertex set of $t_\infty$ and by $BS$ the set of finite binary sequences over the letters $0$ and $1$ including the empty one that we denote by $\emptyset$.
We denote a finite sequence with $n$ elements either as $x_1x_2\cdots x_n$ or as $(x_1,\cdots,x_n).$
We write $x\cdot y$ for the concatenation of two sequences, e.g.~$01\cdot 0011=010011.$ 
Consider the map $bs:\Ver \to BS$ satisfying that the image of the root is the empty sequence.
Moreover, if $\nu\in \Ver$ and $\nu_0,\nu_1$ are the left and right immediate children of $\nu$, then $bs(\nu_0)=bs(\nu)\cdot 0$ and $bs(\nu_1)=bs(\nu)\cdot 1.$
The map $bs$ defines a bijection.
We will often identify a vertex $\nu$ with its associated binary sequence $bs(\nu).$

{\bf Decoration of vertices with sdi.}
By identifying each vertex with its associated binary sequence $bs(\nu)$ we can decorate each vertex with the sdi $I_\nu$ (recall the definition of $I_\nu$ from Subsection \ref{action cantor subsection}).
Here is the beginning of this diagram (to allow for better intuition we have represented each sdi as a subset of the unit interval rather than as a subset of the Cantor space):
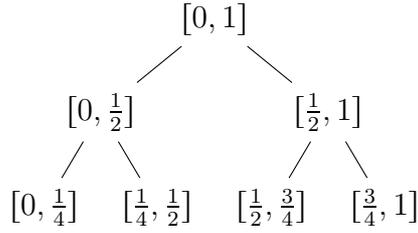
\begin{center}
	\begin{forest}
		[{$[0,1]$} [{$[0, \frac{1}{2}]$}	[{$[0, \frac{1}{4}]$}]
		[{$[\frac{1}{4}, \frac{1}{2}]$}]]
		[{$[\frac{1}{2}, 1]$}	[{$[\frac{1}{2}, \frac{3}{4}]$}]
		[{$[\frac{3}{4}, 1]$}]]]
	\end{forest}
\end{center}

Observe that we have a bijection between the sdi and the vertices of $t_\infty$.

{\bf Disjoint.}
Two vertices are said to be \textit{disjoint} if their associated sdi (as subsets of the Cantor space) are disjoint (otherwise one sdi would be contained in the other). This is equivalent to say that one is not the children of the other.

{\bf Left/right sides, and centre of $t_\infty$.}
The \textit{left (resp. right) side} of $t_\infty$ is the set of vertices whose binary sequence consists only of zeroes (resp. ones). 
The root node neither lies on the left nor right side of $t_\infty$. 
The \textit{centre} of $t_\infty$ is the set of vertices which is not the root node and neither lies on the left or right side of $t_\infty$.

{\bf Tree.}
We call a \textit{tree} any (nonempty) rooted finite subtree of $t_\infty$ so that each vertex has either two immediate children or none. 
A vertex with no immediate children is called a \textit{leaf}.
Moreover, we index increasingly the leaves of a tree from left to right starting at 1. 
We write $\fT$ for the collection of all trees and $\tar(t)$ for the number of leaves of $t$ where $\tar$ stands for ``target''. 
We write $\Leaf(t)$ for the set of leaves of a tree $t$ so that $\tar(t)=|\Leaf(t)|.$

{\bf Forest.}
We call a \textit{forest} a finite union of trees that we represent as finitely many trees placed next to each other ordered from left to right having all roots and leaves lining on two horizontal lines.
The roots and leaves of the forest are ordered from left to right and indexed by natural numbers starting at 1.
Here is an example of a forest with two roots and seven leaves:

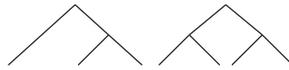
\begin{center}
	\begin{tikzpicture}[baseline=0cm, scale = .8]
		\draw (0,0)--(-1.1, -1);
		\draw (0,0)--(1.1, -1);
		\draw (.55, -.5)--(.05, -1);
	\end{tikzpicture}%
	\hspace*{.5em}%
	\begin{tikzpicture}[baseline=0cm, scale = .8]
		\draw (0,0)--(-.6, -.5);
		\draw (0,0)--(.6, -.5);
		\draw (-.6, -.5)--(-1.1, -1);
		\draw (-.6, -.5)--(-.1, -1);
		\draw (.6, -.5)--(.1, -1);
		\draw (.6, -.5)--(1.1, -1);
	\end{tikzpicture}%
\end{center}

We write $\fF$ for the collection of all forests and $\Root(f),\Leaf(f)$ for the sets of roots and leaves of a forest $f$, respectively.

{\bf Composition of forests.}
When $f,g$ are forests so that the number of leaves of $f$ is equal to the number of roots of $g$ (i.e.~$|\Leaf(f)|=|\Root(g)|$), we consider the \textit{composition} of $f$ with $g$, written $g\circ f$, as the forest obtained by stacking vertically $g$ on the bottom of $f$ lining up the $j$th root of $g$ with the $j$th leaf of $f$.
This provides a partially defined associative binary operation on $\fF$ (and in fact confers a structure of category).

{\bf Partially ordered set.}
We equip $\fT$ with a poset (partially ordered set) structure $\preceq$ that is $t\preceq s$ if $t$ is a rooted subtree of $s$. 
From the description of $\fT$ as rooted subtrees of $t_\infty$ it is obvious that $(\fT,\preceq)$ is directed, i.e.~if $t,s\in\fT$, then there exists $r\in\fT$ satisfying $t\preceq r$ and $s\preceq r$.
Moreover, given $t,s\in\fT$, we have $t\preceq s$ if and only if there exists $f\in\fF$ satisfying $s=f\circ t.$

{\bf From trees to sdp.}
If $t$ is a tree, then for each of its leaves $\ell$ we have an associated sdi $I_\ell$.
Observe that the collection  $(I_\ell, \ \ell\in \Leaf(t))$ is a sdp which we write $sdp(t)$.
Hence, if $\{\nu_i\}_{i \in X} = \Leaf(t)$ from some tree $t$, then we say $\{\nu_i\}_{i \in X}$ is a sdp of vertices which correspond to $sdp(t)$.
The map $t\mapsto sdp(t)$ defines an isomorphism of posets from trees to sdp so that $t\preceq s$ implies that $sdp(s)$ is a refinement of $sdp(t)$.

\textbf{Tree-diagrams.}
A \textit{tree-diagram} is an ordered pair of trees $(t,s) \in \fT \times \fT$ such that $\vert \Leaf(t) \vert = \vert \Leaf(s) \vert$.
From the description of $F$ as acting on the Cantor space and the observation above we deduce that any element $g\in F$ is described by a tree-diagram.
This is not a one to one correspondence since clearly $(f\circ t,f\circ s)$ and $(t,s)$ correspond to the same Thompson's group element.

{\bf Description of $F$ using trees:}
We now provide the 	description of $F$ using tree-diagrams which is due to Brown \cite{brown1987}.
Consider the set $\cP$ of all tree-diagrams $(t,s)$. 
Let $\sim$ be the equivalence relation generated by $(f\circ t,f\circ s)\sim (t,s)$ where $f$ is a forest having the same number of roots as the number of leaves of $t$ (and thus of $s$).
Let $[t,s]$ be the class of $(t,s)$ in the quotient space $\cP/\sim$.
Define the binary operation 
$$[t,s]\circ [t',s']:=[f\circ t, f'\circ s'] \text{ for } f,f'  \text{ satisfying } f\circ s = f'\circ t'.$$
This binary operation is well-defined and confers a group structure to $\cP/\sim$ so that $[t,s]^{-1}=[s,t]$ and $[t,t]$ is the identity for each tree $t$.
Moreover, this group is isomorphic to Thompson's group $F$.

{\bf Reduced pair.}
Consider $g=[t,s]\in F$. A representative $(t,s)$ of $g$ is called \textit{reducible} (and \textit{irreducible} otherwise) if there exists $s',t',f$ with $f$ non-trivial so that $t=f\circ t'$ and $s=f\circ s'$. 
Note: in that case $(t',s')$ is again a representative of $g$.

{\bf Corresponding leaves.}
For a tree-diagram $(t,s)$, two leaves $\nu \in \Leaf(t), \omega \in \Leaf(s)$ are said to correspond to each other if they have the same numbered position, i.e~if $\nu$ is the $j$th leaf of $t$, then $\omega$ is the $j$th leaf of $s$.
By the second description of $F$, the element in $F$ associated with $[t,s]$ maps the sdi $I_\nu$ to $I_\omega$.

{\bf We will be mostly working with this description of $F$ and will thus consider pairs of trees for describing Thompson's group elements. }
Note that one can similarly define $T$ and $V$ using tree-diagrams in a similar way by considering permutations of leaves of trees.

\subsubsection{Description of $F$ using categories}
Another way to define $F$ is to consider the collection of all forests as a (small) category using the composition (i.e.~a set equipped with a partially defined associative binary operation and having an identity for each unit).
This category is cancellative and any pair of forests with the same number of leaves admit a left-common multiple  (this is known as Ore's property).
Hence, $\fF$ embedds in its (left-)groupoid of fractions (i.e.~we can formally manipulate inverses of forests obtaining a groupoid and moreover all element of the groupoid can be written (non-uniquely) as $f^{-1}\circ g$ with $f,g$ forests).
Now, by restricting to pairs of {\it trees} (rather than {\it forests}) we obtain $F$.
Hence, $F$ is the set of $t^{-1}\circ s$ with $t,s$ trees having the same number of leaves equipped with the composition of $\fF$ extended to formal inverses. 

\subsection{Particular trees and forests}

{\bf Particular trees.}
We use the symbol $\Y$ for the tree with two leaves which is equal to a caret.
The trivial tree is the tree having only one leaf (that is equal to its root) that we denote by $e$ or $I$.
We write $t_n$ for the regular tree with $2^n$ leaves all at distance $n$ from the root for the usual tree-metric.

{\bf Tensor product.}
Given two forests $f,g$ we define the \textit{tensor product} $f\ot g$ of them which is a forest obtained by concatenating horizontally $f$ to the left of $g$. 
Note that this is an associative binary operation on the set of forests $\fF$.
As a side remark, if we add the empty forest to $\fF$ (which is the neutral element for $\ot$), then we obtain that $(\fF,\circ,\ot)$ has a structure of a monoidal category.

{\bf Elementary forests.}
An \textit{elementary forest} is of the form $f_{k,n}$ where $1\leq k\leq n$ so that $f_{k,n}$ has $n$ roots, $n+1$ leaves, and all the trees of $f_{k,n}$ are trivial except the $k$th tree which has two leaves.
Using tensor product notation we get $f_{k,n}= I^{\ot k-1} \ot \Y \ot I^{\ot n-k}.$
If the context is clear we may write $f_{k}$ rather than $f_{k,n}.$
The set of elementary forests generates $\fF$ in the sense that any forest is a finite composition of elementary forests. Note that such a decomposition of forests is in general not unique.

\begin{remark}The single tree $\Y$ tensor-generates $\fF$ in the sense that: the smallest subset of $\fF$ that contains $\Y$, the trivial forests, and which is closed under composition and tensor product is equal to $\fF$.
	This is at the root of Jones' technology: one can construct an action of $F$ by defining how $\Y$ ``acts'', see \cite{BSurvey} for details.
\end{remark}

\subsection{Definition of Pythagorean representations} \label{sec:def-pyth}

We explain a particular case of Jones' technology allowing us to construct unitary representations of $F$ (and in fact all these representations extend to the larger Thompson's group $V$) using a pair of operators.

\subsubsection{Pythagorean pairs of operators}
Consider a Hilbert space $\fH$ and two (bounded linear) operators $A,B\in B(\fH)$.
We say that $(A,B)$ is a Pythagorean pair (over $\fH$) if:
$$A^*A + B^* B = \id_\fH$$
where $\id_\fH$ is the identify operator of $\fH$ and $A^*$ is the adjoint of $A$. Define the following corresponding universal $C^*$-algebra.

\begin{definition} \label{universal pythag algebra definition}
	The Pythagorean algebra $P=P_2$ is the universal $C^*$-algebra with generators $a,b$ and the unique relation
	\[a^*a + b^*b = 1.\]
\end{definition}

\subsubsection{Construction of a Hilbert space}

{\bf Hilbert spaces associated to trees.}
Fix a Pythagorean pair $(A,B)$ over $\fH$. 
For each tree $t\in\fT$ we define $\fH_t$ to be $\ell^2(\Leaf(t),\fH)$ the vector space of maps $\xi:\Leaf(t)\to \fH,\ell\mapsto\xi_\ell$ equipped with the inner product:
$$\langle \xi,\eta\rangle = \sum_{\ell\in\Leaf(t)}\langle \xi_\ell, \eta_\ell\rangle.$$
Note that $\fH_t$ is isomorphic to the $n$th direct sum $\fH^n:=\fH^{\oplus n}$ where $n$ is the number of leaves of $t$.
We will make this identification writing $(\xi_1,\cdots,\xi_n)$ the list of values taken by $\xi$ so that $\xi_i$ corresponds to the $i$th leaf of $t$, $1\leq i\leq n.$
To emphasise the tree $t$ and to avoid confusions we may write $(t, \xi)$ or $(t, \xi_1,\cdots,\xi_n)$ rather than $\xi$ or $(\xi_1,\dots,\xi_n).$
Another common way to think of $\xi\in\fH_t$ is to consider the tree $t$ so that each of its leaf $\ell$ is decorated by $\xi_\ell$. 
Hence, $\fH_t$ is the Hilbert space equal to all possible decorations of the leaves of $t$ with elements in $\fH$.
As an example, below is an element of $\fH_t$ where $t$ is the unique tree with two leaves and has been decorated with the vector $(\xi_1, \xi_2)$.

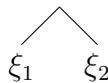
\begin{center}
	\begin{tikzpicture}[baseline=0cm]
		\draw (0,-.3)--(-.5, -.8);
		\draw (0,-.3)--(.5, -.8);
		
		\node[label={[yshift=-22pt] \normalsize $\xi_1$}] at (-.5, -.8) {};
		\node[label={[yshift=-22pt] \normalsize $\xi_2$}] at (.5, -.8) {};
	\end{tikzpicture}%
\end{center}

{\bf Directed system of Hilbert spaces.}
We consider now the family $(\fH_t:\ t\in\fT)$. 
We want to equip this family with a directed set structure.
We start by considering a certain family of isometries.
Let $f_{k,n}=I^{\ot k-1}\ot \Y \ot I^{\ot n-k}$ be an elementary forest with $1\leq k\leq n.$
Define the map: 
\begin{align*}
	\Phi(f_{k,n})=\Phi_{A,B}(f_{k,n}):& \fH^n\to \fH^{n+1},\\
	& (\xi_1,\cdots,\xi_n)\mapsto (\xi_1,\cdots,\xi_{k-1}, A(\xi_k), B(\xi_k), \xi_{k+1},\cdots,\xi_n).
\end{align*}
Note that since $(A,B)$ is a Pythagorean pair we have that $\Phi(f_{k,n})$ is an isometry.

Now, any forest $f$ is a finite product of elementary forests $f=f_{k_1,n_1}\circ\cdots\circ f_{k_r,n_r}$.
We put $\Phi(f):=\Phi(f_{k_1,n_1})\circ\cdots\circ \Phi(f_{k_r,n_r})$.
This is well-defined (if $f$ is written as a different composition of elementary forests we still obtain the same operator). 
Note that $\Phi(f)$ is an isometry for any forest $f$ since it is the composition of some isometries.

{\bf Connecting maps and limits.}
Consider now $t,s\in\fT$ satisfying $t\preceq s$.
There exists a unique forest $f\in\fF$ satisfying $s=f\circ t$.
Define the map 
$$\iota_t^s:\fH_t\to \fH_{s},\ (t,\xi)\mapsto (s,\Phi(f)(\xi)).$$
We obtain a directed system $(\fH_t, \iota_t^s:\ t,s\in\fT, t\preceq s)$ of Hilbert spaces and isometries.
Its directed limit (or inductive limit or colimit) $\scrK=\scrK_{A,B}:=\varinjlim_{t\in\fT} \fH_t$ is a preHilbert space that we complete into a Hilbert space $\scrH=\scrH_{A,B}$.

{\bf Diagrammatic description of the limit space.}
Here is a convenient way to think about the limit space $\scrK$. 
Consider a tree $t$ and index its leaves by some elements in $\fH$. 
We obtain an element $(t,\xi)\in \fH_t$. 
If we add a caret to $t$ at the $k$th leaf we obtain a larger tree $t'\geq t$ that we decorate with $\xi'$ just as $t$ except that the $k$th leaf is decorated by $A(\xi_k)$ and the $(k+1)$th by $B(\xi_k).$
Put $(t,\xi)\sim (t',\xi')$ and consider the smallest equivalence relation generated by it. 
The space $\scrK$ is equal to all decorated trees quotiented by $\sim.$
Hence, any element of $\scrK$ admits a representative $(t,\xi)$ and write $[t,\xi]$ for the class associated to it.
Below we provide an example of the equivalence relation $\sim$ on $\scrK$.
\begin{center}
	\begin{tikzpicture}[baseline=0cm]
		\draw (0,-.3)--(-.5, -.8);
		\draw (0,-.3)--(.5, -.8);
		
		\node[label={[yshift=-22pt] \normalsize $\xi_1$}] at (-.5, -.8) {};
		\node[label={[yshift=-22pt] \normalsize $\xi_2$}] at (.5, -.8) {};
		
		\node at (1.1, -.7) {$\sim$};
	\end{tikzpicture}%
	\begin{tikzpicture}[baseline=0cm]
		\draw (0,0)--(-.5, -.5);
		\draw (0,0)--(.5, -.5);
		\draw (-.5, -.5)--(-.9, -1);
		\draw (-.5, -.5)--(-.1, -1);
		
		\node[label={[yshift=-22pt] \normalsize $A\xi_1$}] at (-.9, -1) {};
		\node[label={[yshift=-22pt] \normalsize $B\xi_1$}] at (-.1, -1) {};
		\node[label={[yshift=-22pt] \normalsize $\xi_2$}] at (.5, -.5) {};		
	\end{tikzpicture}%
\end{center}

For all trees $t$ the direct sum $\fH_t$ naturally embeds inside $\scrK$ given by $\fH_t \ni \xi \mapsto [t, \xi] \in \scrK$. In the sequel, if $\xi \in \fH$ we shall commonly identify $\xi$ with its image $[e, \xi]$ inside $\scrK$.

\begin{remark} \label{scrH dimension remark}
	It is important to note that even when $\fH$ is finite-dimensional, the larger Hilbert space $\scrH$ will always be infinite-dimensional. This can be easily observed by considering the above diagrammatic description of $\scrK$. Indeed, take any non-zero vector $\xi \in \fH$ and for $n \in \N^*$ define $z_n$ to be the element in $\scrK$ formed by indexing the second leaf of $t_n$ with $\xi$ and decorating all the other leaves with zeros. It can then be observed that $\{z_n\}_{n \in \N^*}$ forms an orthogonal set in $\scrH$.
\end{remark}

\subsubsection{The Jones representation associated to a Pythagorean pair}

As before consider a Pythagorean pair $(A,B)$ over a Hilbert space $\fH$, the associated directed system $(\fH_t,\iota_t^s:\ t,s\in\fT, t\preceq s)$, the isometries $(\Phi(f):\ f\in\fF)$ and the limit spaces $\scrK$ and $\scrH$.
We want to construct a (unitary) representation of $F$ on $\scrH$.
Consider $g\in F$ and $z \in \scrK$.
There exists some trees $s,t,r\in\fT$ and a vector $\xi\in\fH_r$ so that $g=[t,s]$ and $z =[r,\xi]$.
Since $(\fT,\preceq)$ is directed there exists $w\in\fT$ so that $s\preceq w$ and $r\preceq w$ and thus some forests $f,h$ satisfying $w=f\circ s = h\circ r$.
Using the equivalence relations defining $F$ and $\scrK$ we have: 
$$g=[t,s]=[f\circ t,f\circ s]=[f\circ t,w] \text{ and } z=[r,\xi]=[h\circ r, \Phi(h)(\xi)] = [w,\Phi(h)(\xi)].$$
We set: 
$$\sigma(g)z:=\sigma([f\circ t,w]) [w,\Phi(h)(\xi)] := [f\circ t,\Phi(h)(\xi)].$$
This defines an action by unitary operators of $F$ on $\scrK$ which extends into a unitary representation $\sigma=\sigma_{A,B}$ acting on $\scrH$.
This is indeed well-defined (the formula does not depend on the choice of the representatives $(t,s)$ and $(r,\xi)$ nor depend on the choice of $f,h$).

\begin{definition}We call $\sigma_{A,B}$ the \textit{Jones representation or Jones action} associated to the Pythagorean pair $(A,B)$ or the \textit{Pythagorean representation} associated to $(A,B)$.
\end{definition}

{\bf Visualisation of the Jones action.}
Here is one key example to keep in mind for visualising $\sigma$.
Consider $g=[t,s]\in F$ and $z=[s,\xi]\in \scrH$ where we have taken the second tree $s$ of the representative of $g$ to be equal to the tree of the representative of $z$.
The element $z$ corresponds to the (class of the) tree $s$ with leaves decorated by $\xi=(\xi_1,\dots,\xi_n)$ where $n$ is the number of leaves of $s$ (and thus of $t$).
Now, $\sigma(g)(z)=[t,\xi]$ corresponds to the tree $t$ (instead of $s$) with the same decoration $\xi$ of leaves.
Hence, the Jones action did not change the decoration but only the tree.
An example of Jones' action is shown below.

\begin{center}
	\begin{tikzpicture}[baseline=0cm]
		\draw (0,0)--(-.5, -.5);
		\draw (0,0)--(.5, -.5);
		\draw (-.5, -.5)--(-.9, -1);
		\draw (-.5, -.5)--(-.1, -1);
		
		\node[label={\normalsize $\sigma($}] at (-1.1, -1) {};
		\node[label={\normalsize $,$}] at (.65, -1) {};
	\end{tikzpicture}%
	\begin{tikzpicture}[baseline=0cm]
		\draw (0,0)--(-.5, -.5);
		\draw (0,0)--(.5, -.5);
		\draw (.5, -.5)--(.1, -1);
		\draw (.5, -.5)--(.9, -1);
		
		\node[label={\normalsize $)$}] at (1.1, -1) {};
		\node[label={\normalsize $\cdot$}] at (1.35, -.9) {};
	\end{tikzpicture}%
	\begin{tikzpicture}[baseline=0cm]
		\draw (0,0)--(-.5, -.5);
		\draw (0,0)--(.5, -.5);
		\draw (.5, -.5)--(.1, -1);
		\draw (.5, -.5)--(.9, -1);
		
		\node[label={[yshift=-22pt] \normalsize $\xi_1$}] at (-.5, -.5) {};
		\node[label={[yshift=-22pt] \normalsize $\xi_2$}] at (.1, -1) {};
		\node[label={[yshift=-22pt] \normalsize $\xi_3$}] at (.9, -1) {};	
		
		\node[label={\normalsize $=$}] at (1.35, -1) {};
	\end{tikzpicture}%
	\begin{tikzpicture}[baseline=0cm]
		\draw (0,0)--(-.5, -.5);
		\draw (0,0)--(.5, -.5);
		\draw (-.5, -.5)--(-.9, -1);
		\draw (-.5, -.5)--(-.1, -1);
		
		\node[label={[yshift=-22pt] \normalsize $\xi_1$}] at (-.9, -1) {};
		\node[label={[yshift=-22pt] \normalsize $\xi_2$}] at (-.1, -1) {};
		\node[label={[yshift=-22pt] \normalsize $\xi_3$}] at (.5, -.5) {};		
	\end{tikzpicture}%
\end{center}

{\bf Categorical interpretation.}
Initially, Jones' actions were defined using categories and functors as we are about to explain. 
This description will not be used or referred later in the paper. It is here for the curiosity of the reader. 
Consider the monoidal category of binary forests $(\fF,\circ,\otimes)$ where the objects of $\fF$ are the natural numbers, the morphisms the forests, the composition $\circ$ the vertical concatenation, and the monoidal structure $\ot$ the horizontal concatenation.
Consider now the category of Hilbert spaces with isometries for morphisms and equip it with the monoidal structure given by the \textit{direct sum} of Hilbert spaces (hence not the usual monoidal structure given by the Hilbert space tensor product).
Now, choosing a Pythagorean pair $(A,B)$ corresponds in choosing a covariant monoidal functor from the category of binary forests to the category of Hilbert spaces. 
The functor being the $\Phi$ we used above satisfying $\Phi(\Y)(\xi) = (A\xi,B\xi)\in \fH\oplus\fH$ for $\xi\in\fH:=\Phi(1).$
Jones' technology can be applied to construct unitary representations of $F$ using \textit{any} functor from $\fF$ to the category of Hilbert spaces (not necessarily covariant nor monoidal). 
Hence, we are considering a very special kind of those.
Moreover, it provides an action of the Thompson groupoid (the whole fraction groupoid obtained from $\fF$) that we restrict in this paper to the Thompson group.

%%%%%%%%%%%END PRELIMINARIES%%%%%%%%%%%%%%%%%%%%%%%%%%%%%%%%%%%%
%%%%%%%%%%%%%%%%%%%%%%%%%%%%%%%%%%%%%%%%%%%%%%%%%%%%%%%%%%

\section{Tools and general results for Pythagorean representations}\label{pythag rep section}

In this section we fix a Pythagorean pair $(A,B)$ acting on $\fH$ and consider the associated representation $(\sigma,\scrH)$ of $F$.
We recall that $\scrK\subset\scrH$ is the dense subspace of equivalence classes $[t,\xi]$ of trees $t$ whose leaves are decorated by vectors of $\fH$.

\subsection{Definitions of some partial isometries} \label{partial isom section}

\textbf{Partial isometries associated to a vertex.}

Consider a tree $t$ and a vector $[t,\xi]\in \scrK$: that is the tree $t$ with its leaves $\ell$ decorated by elements $\xi_\ell$ of $\fH$.
Given any vertex $\nu$ of the infinite binary rooted tree we want to define the {\it $\nu$-component} of $[t,\xi]$.
\begin{itemize}
	\item If $\nu$ is a leaf $\ell$ of $t$, then it is $\xi_\ell\in\fH$ that we interpret as an element of $\scrH$. 
	\item If $\nu$ is not a vertex of $t$, then we choose another representative $(t',\xi')$ of $[t,\xi]$ such that $\nu$ is a leaf $\ell'$ of $t'$ and take $\xi'_{\ell'}$. Note that $t'$ is necessarily of the form $f\circ t$ where $f$ is a forest we attach on the bottom of $t$.
	\item If $\nu$ is an interior vertex of $t$ (a vertex of $t$ that is not a leaf), then we consider $t_\nu$ the subtree of $t$ rooted at $\nu$ and having same children of $\nu$ than $t$. In particular, the leaves of $t_\nu$ forms a subset of the leaves of $t$. The $\nu$-component is then $[t_\nu,\eta]\in\fH_{t_\nu}$ deduced from the subtree $t_\nu$ of $t$ and the decoration of leaves being the restriction of the one of $t$ that is $\eta:\Leaf(t_\nu)\ni \ell\mapsto \xi_\ell$.
\end{itemize}

Intuitively, taking the $\nu$-component consists of ``growing'' the tree large enough and then ``snipping'' the tree at the vertex $\nu$. We formalise this notion in the following proposition.

\begin{proposition}
	For any $\nu\in\Ver$, taking the $\nu$-component is a well-defined map from $\scrK$ to $\scrK$ that extends into a surjective partial isometry
	$$\tau_\nu:\scrH\to \scrH.$$
\end{proposition}

\begin{proof}
	It is not difficult to observe that taking the $\nu$-component is a well-defined map from $\scrK$ to $\scrK$ as a consequence of $\Phi$ being a monoidal functor. This map is clearly surjective and linear. Define the subspace $\fX_\nu \subset \scrK$ consisting of all elements such that their $\omega$-component (in the sense as described in the beginning of the subsection) is zero for all vertices $\omega$ which are disjoint from $\nu$ (recall $\nu$ and $\omega$ are disjoint if their associated sdi's are disjoint). 
	Now, if $\xi\in\fX_\nu$, then the $\nu$-component of $\xi$ has same norm than $\xi$.
Further, if $x \in \fX_\nu^\perp \cap \scrK$ then necessarily the $\nu$-component of $x$ must be $0$. Thus the map taking the $\nu$-component is bounded and can be continuously extended to a surjective partial isometry $\tau_\nu$ on $\scrH$ with initial space the norm-closure of $\fX_\nu$.
\end{proof}

\begin{example}
	Consider $[s,\xi]$ where $s=\Y$ is the tree with two leaves, $\xi=(\xi_1,\xi_2)$, and $\nu=01$.
	Note that $(s,\xi)\sim (f_1\circ s, (A\xi_1, B\xi_1, \xi_2))$ and thus the $01$-component of $[s,\xi]$ is $B\xi$ as shown by the diagram.
	
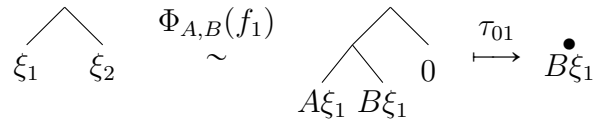
\begin{center}
		\begin{tikzpicture}[baseline=0cm, scale = 1]
			\draw (0,0)--(-.5, -.5);
			\draw (0,0)--(.5, -.5);
			
			\node[label={[yshift=-22pt] \normalsize $\xi_1$}] at (-.5, -.5) {};
			\node[label={[yshift=-22pt] \normalsize $\xi_2$}] at (.5, -.5) {};
			
			\node[label={[yshift= -3pt] \normalsize $\Phi_{A,B}(f_1)$}] at (2, -.7) {$\sim$};
		\end{tikzpicture}%
		\begin{tikzpicture}[baseline=0cm, scale = 1]
			\draw (0,0)--(-.5, -.5);
			\draw (0,0)--(.5, -.5);
			\draw (-.5, -.5)--(-.9, -1);
			\draw (-.5, -.5)--(-.1, -1);
			
			\node[label={[yshift=-22pt] \normalsize $A\xi_1$}] at (-.9, -1) {};
			\node[label={[yshift=-22pt] \normalsize $B\xi_1$}] at (-.1, -1) {};
			\node[label={[yshift=-22pt] \normalsize $0$}] at (.5, -.5) {};
			
			\node[label={[yshift= -3pt] \normalsize $\tau_{01}$}] at (1.4, -.7) {$\longmapsto$};
		\end{tikzpicture}%
		\begin{tikzpicture}[baseline=0cm, scale = 1]
			\node[label={[yshift=-25pt] \normalsize $B\xi_1$}] at (0, -.5) {$\bullet$};
		\end{tikzpicture}%
	\end{center}
\end{example}

\begin{remark} \label{partial isom remark}
	\begin{enumerate}[i]
		\item
		Since $\tau_\nu$ (for a given vertex $\nu$) is a surjective partial isometry we have that its adjoint $\tau_\nu^*$ is an isometry being a right-inverse of $\tau_\nu$. Moreover, note that $\tau_\omega\circ \tau_\nu^*=0$ for all vertex $\omega$ which is disjoint from $\nu$. 
		Graphically, $\tau^*_\nu$ acts on $(t, \xi) \in \fH_t$ by ``lifting'' the tree $t$ and attaching the resulting subtree, along with the components of $\xi$, at the vertex $\nu$ while setting all other components to $0$.
		\item Define the orthogonal projection $$\rho_\nu := \tau_\nu^*\tau_\nu$$ which sets the $\omega$-components of an element to $0$ for all vertices $\omega$ that are disjoint from $\nu$.	If $\{\nu_i\}_{i=1}^n$ is a sdp then $\sum_{i=1}^n \rho_{\nu_i} = \id$.
		The action of $\rho_0$ is shown below.
		
		
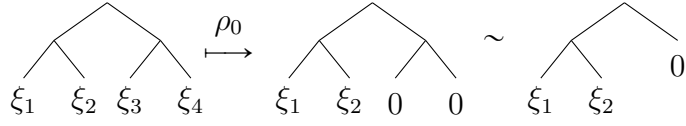
\begin{center}
			\begin{tikzpicture}[baseline=0cm, scale = 1]
				\draw (0,0)--(-.7, -.5);
				\draw (0,0)--(.7, -.5);
				\draw (-.7, -.5)--(-1.1, -1);
				\draw (-.7, -.5)--(-.3, -1);
				\draw (.7, -.5)--(.3, -1);
				\draw (.7, -.5)--(1.1, -1);
				
				\node[label={[yshift=-22pt] \normalsize $\xi_1$}] at (-1.1, -1) {};
				\node[label={[yshift=-22pt] \normalsize $\xi_2$}] at (-.3, -1) {};
				\node[label={[yshift=-22pt] \normalsize $\xi_3$}] at (.3, -1) {};
				\node[label={[yshift=-22pt] \normalsize $\xi_4$}] at (1.1, -1) {};
				
				\node[label={[yshift= -3pt] \normalsize $\rho_{0}$}] at (1.6, -.7) {$\longmapsto$};
			\end{tikzpicture}%
			\begin{tikzpicture}[baseline=0cm, scale = 1]
				\draw (0,0)--(-.7, -.5);
				\draw (0,0)--(.7, -.5);
				\draw (-.7, -.5)--(-1.1, -1);
				\draw (-.7, -.5)--(-.3, -1);
				\draw (.7, -.5)--(.3, -1);
				\draw (.7, -.5)--(1.1, -1);
				
				\node[label={[yshift=-22pt] \normalsize $\xi_1$}] at (-1.1, -1) {};
				\node[label={[yshift=-22pt] \normalsize $\xi_2$}] at (-.3, -1) {};
				\node[label={[yshift=-22pt] \normalsize $0$}] at (.3, -1) {};
				\node[label={[yshift=-22pt] \normalsize $0$}] at (1.1, -1) {};
				
				\node at (1.6, -.5) {$\sim$};
			\end{tikzpicture}%
			\begin{tikzpicture}[baseline=0cm, scale = 1]
				\draw (0,0)--(-.7, -.5);
				\draw (0,0)--(.7, -.5);
				\draw (-.7, -.5)--(-1.1, -1);
				\draw (-.7, -.5)--(-.3, -1);
				
				\node[label={[yshift=-22pt] \normalsize $\xi_1$}] at (-1.1, -1) {};
				\node[label={[yshift=-22pt] \normalsize $\xi_2$}] at (-.3, -1) {};
				\node[label={[yshift=-22pt] \normalsize $0$}] at (.7, -.5) {};		
			\end{tikzpicture}%
		\end{center}		
		
		More generally, $\tau^*_\nu\tau_\omega$ is the partial isometry which ``snips'' the tree at $\omega$ and attaches the resulting subtree, along with its components, at the vertex $\nu$ while setting all other components to $0$. Below shows an example of how the map $\tau^*_{0}\tau_{1}$ acts.
		
		
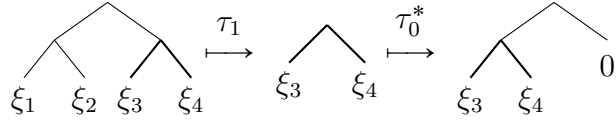
\begin{center}
			\begin{tikzpicture}[baseline=0cm, scale = 1]
				\draw (0,0)--(-.7, -.5);
				\draw (0,0)--(.7, -.5);
				\draw (-.7, -.5)--(-1.1, -1);
				\draw (-.7, -.5)--(-.3, -1);
				\draw[thick] (.7, -.5)--(.3, -1);
				\draw[thick] (.7, -.5)--(1.1, -1);
				
				\node[label={[yshift=-22pt] \normalsize $\xi_1$}] at (-1.1, -1) {};
				\node[label={[yshift=-22pt] \normalsize $\xi_2$}] at (-.3, -1) {};
				\node[label={[yshift=-22pt] \normalsize $\xi_3$}] at (.3, -1) {};
				\node[label={[yshift=-22pt] \normalsize $\xi_4$}] at (1.1, -1) {};
				
				\node[label={[yshift= -3pt] \normalsize $\tau_{1}$}] at (1.6, -.7) {$\longmapsto$};
			\end{tikzpicture}%
			\begin{tikzpicture}[baseline=0cm, scale = 1]
				\draw[thick] (0,-.3)--(-.5, -.8);
				\draw[thick] (0,-.3)--(.5, -.8);
				
				\node[label={[yshift=-22pt] \normalsize $\xi_3$}] at (-.5, -.8) {};
				\node[label={[yshift=-22pt] \normalsize $\xi_4$}] at (.5, -.8) {};
				
				\node[label={[yshift= -3pt] \normalsize $\tau^*_{0}$}] at (1.1, -.7) {$\longmapsto$};
			\end{tikzpicture}%
			\begin{tikzpicture}[baseline=0cm, scale = 1]
				\draw (0,0)--(-.7, -.5);
				\draw (0,0)--(.7, -.5);
				\draw[thick] (-.7, -.5)--(-1.1, -1);
				\draw[thick] (-.7, -.5)--(-.3, -1);
				
				\node[label={[yshift=-22pt] \normalsize $\xi_3$}] at (-1.1, -1) {};
				\node[label={[yshift=-22pt] \normalsize $\xi_4$}] at (-.3, -1) {};
				\node[label={[yshift=-22pt] \normalsize $0$}] at (.7, -.5) {};		
			\end{tikzpicture}%
		\end{center} 
	\end{enumerate}
\end{remark}

\subsection{Some Identities of Pythagorean representations}
Viewing elements of Thompson's group $F$ as a pair of trees in combination with the family of partial isometries $\{\tau_\nu\}_{\nu \in \Ver}$ provides a powerful method for analysing properties of Pythagorean representations. Below we list some obvious but useful identities of Pythagorean representations involving these partial isometries which we shall frequently refer to later in the paper.  
\begin{enumerate}[i]
	
	\item For all sdi $I$ there is a group isomorphism $\alpha:\Fix_F(I^c)\to F$ where $I^c$ is the complement of $I$ and $\Fix_F(I^c)$ is the subgroup of $g\in F$ fixing all points of $I^c$.
	Indeed, there exists a finite word $w$ such that $I=\{ w\cdot x:\ x\in \{0,1\}^\N\}.$ Now, consider the map $w\cdot x\to x$ from $I$ to the whole Cantor space (corresponding to scaling $I$ to $[0,1]$ in the unique affine way if we work over real numbers rather than sequences). It induces the isomorphism of above. Moreover, $\alpha$ extends to a surjection from $\Stab_F(I)$ onto $F$ where $\Stab_F(I)$ stands for the stabiliser subgroup of $I$ in $F$.
	
	\item We now interpret (i) using the bijection between sdi and vertices. 
	Given a vertex $\nu$ and a tree $t$ we consider $t_\nu$ as defined in Section \ref{partial isom section} being the maximal subtree of $t$ rooted at $\nu$.
	If $g=[t,s]$ is in $\Stab_F(I_\nu)$, then $g_\nu := [t_\nu,s_\nu]$ is in $F$ ($t_\nu, s_\nu$ necessarily have the same number of leaves) and corresponds to the affine scaling of $g$ when restricted to $I_\nu$ as defined in (i).	
	
	\item The mapping $$\Stab_F(I_\nu)\to F,\ g\mapsto g_\nu$$ of (ii) is compatible with the Jones action $\sigma$:
	\begin{equation} \label{action to stabaliser equation}
		\sigma(g_\nu) \circ \tau_\nu = \tau_\nu \circ \sigma(g).
	\end{equation}
	Further, if $\{\nu_i\}_{i = 1}^n$ is a sdp such that $g \in \Stab_F(\nu_i)$ for all $i = 1, 2, \dots n$ then:
	\begin{equation} \label{action decomposition}
		\sigma(g) = \sum_{i = 1}^n \tau^*_{\nu_i} \circ \sigma(g_{\nu_i}) \circ \tau_{\nu_i}.	
	\end{equation}
	
	\item Let $z \in \scrH$ and $g := [t, s] \in F$ where $\Leaf(t) := \{\nu_i\}_{i = 1}^n$ and $\Leaf(s) := \{\omega_i\}_{i = 1}^n$. Then for all $ i = 1, 2, \dots, n$:
	\[\tau_{\nu_i} \circ \sigma(g) = \tau_{\omega_i}.\]
	Diametrically, for $z := [t, \xi]$, this can be viewed as the action $\sigma(g)$ takes the \textit{subtree} which has root node at $\omega_i$ and then attaches it to the vertex $\nu_i$ along with the components of $\xi$ which are children of $\omega_i$. This is extended by the following identity:
	\begin{equation} \label{action rearrange equation}
		\sigma(g) = \sum_{i = 1}^n \tau^*_{\nu_i} \tau_{\omega_i}.
	\end{equation}
	
	\item Let $z \in \scrH$ and let $g$ be defined as above. From the above statement, $\sigma(g)z = z$ if and only if:
	\begin{equation} \label{equality vertices equations}
		\tau_{\nu_i}(z) = \tau_{\omega_i}(z),\quad \textrm{for all } i = 1, 2, \dots, n.
	\end{equation}
\end{enumerate} 

\subsection{Diffuse Pythagorean pairs.}
Recall the \textit{strong operator topology} (SOT) of $B(\fH)$ is the locally convex topology obtained from the seminorms $B(\fH)\ni T\mapsto \|T(\xi)\|$ indexed by the vectors $\xi\in\fH$.
In particular, a net of operators $(T_i)_{i\in I}$ converges to $T$ for the SOT, denoted 
$$T_i\xrightarrow{s} T$$ 
when $\lim_{i\in I}\|T(\xi)-T_i(\xi)\|=0$ for all $\xi\in\fH$.

\begin{definition} \label{diffuse definition}
	We call a Pythagorean pair $(A,B)$ acting on $\fH$ to be \textit{diffuse} if 
	$$\lim_{n \rightarrow \infty} p_n\xi=0$$ 
	for all $\xi\in\fH$ and all sequences $(p_n : n \geq 1)$ such that $p_n = x_n\dots x_2x_1$ where each $x_k \in \{A,B\}$. We say that $(p_n)$ is an increasing sequence of words in $A,B$. \\
	A Pythagorean representation is \textit{diffuse} if it is associated to a diffuse Pythagorean pair.
\end{definition}
In other words: $(A,B)$ is diffuse if any increasing chain of words in $A,B$ converges to 0 for the strong operator topology.

\begin{lemma} \label{diffuse lemma}
	Assume that $(A,B)$ is a diffuse Pythagorean pair of operators with associated representation $(\sigma,\scrH)$.
	If $\nu \in \Ver$ is a vertex different from the root and $z\in\scrH$ is a non-zero vector, then $\tau_\nu(z)\neq z$.
\end{lemma}

\begin{proof}
	Let $z \in \scrH$ and $\nu \in \Ver \backslash \{\varnothing\}$ such that $\tau_\nu(z) = z$. Denote $\nu^n$ to be the vertex corresponding to the binary sequence $\nu$ concatenated $n$ times.
	Define $p$ to be the unique ray which passes through the sequence of vertices $(\nu^n : n \geq 1)$ (by a ray we mean an infinite geodesic path in $t_\infty$ which begins from the root node). Note that $p$ is indeed unique because $\nu \neq \varnothing$.
	Observe that $\norm{z} = \norm{\tau_{\nu^n}(z)}$ for all $n \in \N^*$. This implies that $\tau_\omega(z) = 0$ if the vertex $\omega$ does not lie on the ray $p$.
	
	Now, consider an arbitrary vector $x := [t, \xi] \in \scrK$. We are going to show that $x$ is orthogonal to $z$ which will imply that $z=0$ since $\scrK$ is dense in $\scrH$.
	Iteratively define a sequence of representatives of $x$ by $(x^k : k \geq 0)$ such that $x^0 := (t^0, \xi^0) = (t, \xi)$ and:
	\begin{align*}
		x^{k+1} &:= (t^{k+1}, \xi^{k+1}) = (f\circ t^k, \Phi(f)\xi^k)
	\end{align*}
	where $f$ is the elementary forest having a single caret at the unique leaf $\ell_k$ of $t^k$ that lies in the ray $p$.
	Further, define $([\ti x^k] : k \geq 1) \subset \scrK$ by $\ti x^k = (t^k, \ti \xi^k)$ where:
	\begin{equation*}
		\ti \xi^k_\ell =
		\begin{cases*}
			\xi^k_\ell, \quad &$\ell \neq \ell_k,$\\
			0, \quad &$\ell = \ell_k$
		\end{cases*}
	\end{equation*}
	and $\xi_\ell^k$ denotes the component of $\xi^k$ at the leaf $\ell$.
	Hence, $\ti \xi^k$ is the tree $t^k$ with same decoration than $x^k$ except that the leaf $\ell_k$ is decorated by $0$.
	Define $p'$ to be the unique, infinite subpath of $p$ which begins from $\ell_0$ and let $(p'_n : n \geq 1)$ be the sequence of operators induced by $p'$. That is, $p'_n = a_n\dots a_2a_1$ such that $a_k = A$ if the $k$th edge of $p'$ is a left-edge, otherwise $a_k = B$. Using the initial assumption that $(A,B)$ is diffuse we have:
	\begin{align*}
		\xi^{k}_{\ell_k} & = p'_k\xi^0_{\ell_0} = p'_k\xi_{\ell_0} \rightarrow 0\ \textrm{as}\ k \rightarrow \infty.
	\end{align*}
	Hence, by choosing a sufficiently large $m$ we can make $\norm{x - [\ti x^m]} = \norm{\xi^m - \ti \xi^m} = \vert \xi^m_{\ell_m} \vert$ arbitrarily small. Therefore, $x$ is the norm limit of the sequence $([\ti x^n] : n \geq 1)$. Further, observe by construction that each $[\ti x^n]$ is orthogonal to the vector $z$. Therefore we have:
	\[\langle x, z \rangle = \lim_{n\to\infty} \langle [\ti x^n], z \rangle  = 0.\]
	Since $\scrK$ is a dense subspace $\scrH$ this implies $z \in \scrH^\perp$ and thus $z = 0$. The reverse direction of the lemma is obvious. This completes the proof. 
\end{proof}

\begin{remark}
	The above lemma shows that if $(\sigma,\scrH)$ is diffuse and $z\in\scrH$ is non-zero, then necessarily $\tau_\nu(z)\neq 0$ on a set of $\nu$ that is not contained in finitely many rays. Intuitively, $z$ is supported on a diffuse set of vertices of the infinite binary tree which motivates our terminology.
\end{remark}

\textbf{For the remainder of the paper we shall focus on diffuse Pythagorean pairs of operators.}

\subsection{Von Neumann ergodic theorem and consequences for partial isometries.}
We deduce an interesting consequence of the von Neumann ergodic theorem (\cite{neumann1932proof}, see \cite[Theorem II.11]{reed1972methods} for a more recent proof) that links our projections $\rho_\nu$ to the Jones action $\sigma$.

\begin{theorem} \label{MET} [von Neumann ergodic theorem]
	Let $H$ be a Hilbert space, $U$ a unitary operator acting on it, and $N:=\ker(U-\id_H)\subset H$ the vectors fixed by $U$.
	We have that 
	\[\frac{1}{n}\sum_{k=0}^{n-1}U^k \xrightarrow{s} \textrm{proj}_N\]
	where $\textrm{proj}_N$ is the orthogonal projection onto $N.$
\end{theorem}

Given a vertex $\nu$ recall that $\rho_\nu:=\tau_\nu^*\circ \tau_\nu$ is the associated projection as defined in Section \ref{partial isom section} and $I_{\nu}$ is the associated sdi. If $I$ is a finite disjoint union of sdi's $\{I_{\nu_i}\}_{i=1}^n$ (like the support of a $g\in F$), then define the projection: 
$$\rho_I := \sum_{i=1}^n \rho_{\nu_i}.$$
Here is a very useful consequence of the von Neumann ergodic theorem applied to Pythagorean representations.

\begin{proposition} \label{MET rho}
Let $(A,B)$ be a diffuse Pythagorean pair with associated Jones' representation $(\sigma,\scrH)$. 
Then for all $g \in F$ we have:
\begin{align*} \ker (\sigma(g) - \id) & = \{z \in \scrH : \tau_\nu(z) = 0 \text{ for all } \nu \in \Ver \text{ satisfying } I_\nu \subset \supp(g)\}\\
& = \bigcap_{\nu:\ I_\nu\subset \supp(g)} \ker(\tau_\nu)
\end{align*} 
and
$$\frac{1}{n}\sum_{k=0}^{n-1}\sigma(g^k) \xrightarrow{s} \id - \rho_{\supp(g)}.$$
\end{proposition}

\begin{proof}
	Consider $A,B,\scrH,\sigma$ as above and $g\in F$.
	It is easy to see from Equation \ref{action decomposition} that the following holds:
	\[\{z \in \scrH : \tau_\nu(z) = 0, \forall \nu \in \Ver : I_\nu \subset \supp(g)\} \subset \ker(\sigma(g) - \id).\]
	Thus, we only need to prove the reverse inclusion. That is, if $z \in \scrH$ and $\sigma(g)z = z$, then $\tau_\nu(z) = 0$ for all $\nu \in \Ver$ such that $I_\nu \subset \supp(g)$. Therefore, again from Equation \ref{action decomposition}, it is sufficient to only consider $g \in F$ with full support and prove that $\ker(\sigma(g) - \id) = \{0\}$. 
	Consider such an element $g := [t, s] \in F$ and assume there exists a non-zero vector $z$ fixed by $\sigma(g)$. 
	Let $(\nu_1,\dots,\nu_n)$ and $(\omega_1,\dots,\omega_n)$ be the leaves of $t$ and $s$, respectively.
	By Equation \ref{equality vertices equations} we have $\tau_{\nu_i}(z) = \tau_{\omega_i}(z)$ for all $1\leq i \leq n$.
	We conclude by using the assumption.
	
	Indeed, let $j$ be the first index satisfying $\tau_{\omega_j}(z)\neq 0$ and thus the first index satisfying $\tau_{\nu_j}(z)\neq 0$ as well. Such an index exists because $z\neq 0$.
	This implies $\nu_j$ and $\omega_j$ are not disjoint vertices (i.e.~their associated sdi have non-trivial intersection). Further, $\nu_j$ and $\omega_j$ do not coincide because $g$ has full support. Hence, up to switching the roles of $t$ and $s$ (which corresponds to taking the inverse of $g$) we can assume that $\nu_j$ is a child of $\omega_j$.
	Then the binary sequence of $\nu_j$ is obtained by concatenating $\omega_j$ with a non-trivial word $\mu$. This implies that $\tau_\mu(\tau_{\omega_j}(z)) = \tau_{\omega_j}(z)\neq 0$ contradicting Lemma \ref{diffuse lemma}.
	
	The second statement follows immediately from the first and Theorem \ref{MET} by noting that $\id-\rho_{\supp(g)}$ is the orthogonal projection onto the subspace $\ker(\id-\sigma(g)) \subset \scrH$.
\end{proof}

\begin{remark}
	The proof of Proposition \ref{MET rho} can be slightly modified to show that for all $g \in F$, there is no $z \in \scrH$ and $\lambda \in \bS$ with $\lambda \neq 1$ such that $\sigma(g)z = \lambda z$. In particular, this proves that if $g$ has full support then $\sigma(g)$ has no eigenvalues, otherwise $\sigma(g)$ only has the single eigenvalue $1$.
\end{remark}

%%%%%%%%%%%%END TOOLS FOR PYTH REP%%%%%%%%%%%%%%%%%%%%%%%
%%%%%%%%%%%%%%%%%%%%%%%%%%%%%%%%%%%%%%%%%%%%%%%

\section{Containment of Induced Representations} \label{induced chapter}
In the quest to find families of \textit{irreducible} representations of Thompson's groups, one obvious source are induced representations associated to self-commensurator subgroups of Thompson's groups (the irreducibility is a consequence of the Mackey-Shoda criteria \cite{Mack51}, see also \cite[Theorems 1.F11 and 1.F.17]{BH} and \cite{Burger-Harpe97}).
These representations are easily constructible without using Jones' technology and have somehow too obvious finite-dimensional roots. Hence, we are interested in obtaining representations far from those.
This leads to the following notion and question.

\begin{definition}
	Consider a unitary representation $\sigma$ of a discrete group $G$.
	We say that $\sigma$ is \NInd\ if given any non-trivial subgroup $K\subset G$ and any finite-dimensional non-zero unitary representation $\theta:K\act\fK$ we have that $\sigma$ does not contain the induced representation $\Ind_K^G\theta$.
\end{definition}

\begin{question} \label{contain monomial rep question}
	Let $\sigma$ be a diffuse Pythagorean representation. Do we have that $\sigma$ is \NInd?
\end{question}

Note that a mixing representation (e.g~the regular one $\la_F:F\act \ell^2(F)$) is always \NInd\ but for some rather trivial reason, see Remark \ref{obs:NInd-mixing}. Although, none of the Pythagorean representations is mixing (see the same remark) making our question far from being obvious.

As far as the authors are aware of, from previously constructed families of representations of Thompson's groups, only Garncarek has (partially) addressed the above question \cite{garncarek2012analogs}. Indeed, he constructed a family of representations of $F$ and showed that they are not the induction of a finite-dimensional representation of a {\it parabolic} subgroup (i.e.~the subgroups $F_p:=\{g\in F:\ g(p)=p\}$ for some ray $p$ in the Cantor space).

In this section we shall indeed answer positively to our question proving the main result of the paper. 
This extends and strengthens the result of Garncarek. 
Indeed, there is a large wealth of subgroups of $F$ which are not fixed point subgroups (or at least are not known to be of this form); to cite a few examples: the derived subgroup, the rectangular (finite index) subgroups of Bleak-Wassink \cite{bleak2007finite}, the oriented Jones-Thompson subgroup $\Vec F\subset F$ \cite{Jones17} and analogous diagrammatically constructed subgroups (providing infinite index maximal subgroups) \cite{ren2018skein, aiello2021}, as well as the maximal subgroups of \cite{golan2016generation, golan2017subgroups}.
Moreover, the representations of Garncarek are all Pythagorean (arising from one-dimensional $\fH$), see Section \ref{sec:example}.
Finally, as far as we are aware of, this result produces the first known representations of $F$ that are \NInd\ but not mixing, see Remark \ref{obs:NInd-mixing}.

\begin{theorem} \label{non-induced rep theorem}
	All diffuse Pythagorean representations are \NInd.
\end{theorem}

\begin{proof}
	Consider a diffuse Pythagorean pair $(A,B)$ acting on $\fH$ and write $(\sigma,\scrH)$ for the associated Pythagorean representation.
	Suppose there exists a non-trivial subgroup $H \subset F$ and a non-zero finite-dimensional representation $\theta:H\act\fK$ of $H$ such that $\sigma \supset \Ind_H^F \theta$. 
	Up to identification, we may assume that $\fK$ is a subspace of $\scrH$. 
	Moreover, note that $\sigma(g)(\fK)\perp \fK$ for all $g\notin H$ implying that 
	\[\phi_z(g) := \langle \sigma(g)z, z \rangle = 0 \text{ for all } g\in F\backslash H \text{ and } z\in \fK.\]
		
	Our general strategy is to prove that $\sigma\restriction_H$ is weakly mixing by proving that $H$ contains an infinite subset $H'$ with good asymptotic properties. This will lead to a contradiction since a weakly mixing representation does not contain any non-trivial finite-dimensional subrepresentation.
		
	{\bf Notation.} From now on we fix a unit vector $z\in\fK$. 
	
	\textbf{(1)} First, we shall prove the following claim.
	
	\textbf{Claim 1:} If there exists $g \in F$ such that $\{g^n\}_{n \in \N^*} \subset F \backslash H$, then $\rho_N(z) = 0$ where 
	$$N := \supp(g)^c=\cC\setminus \supp(g)$$
	and recall $\cC$ denotes the Cantor space.
	Equivalently, $\fK$ is in the range of $\rho_{\supp(g)}$ for all $g\in F$ satisfying that $g^n\notin H$ for all $n\geq 1.$
	
	Suppose $g \in F$ satisfying that $g^n\notin H$ for all $n\geq 1$.
	Proposition \ref{MET rho} implies that:
	\begin{align*}
		\norm{\rho_N(z)}^2 &= \langle \rho_N(z), \rho_N(z) \rangle  = \langle \rho_N(z) , z\rangle\\
		& = \biggl \langle \lim_{n \rightarrow \infty} \frac{1}{n}\sum_{k=0}^{n-1}\sigma(g^k)z, z \biggr \rangle 
		= \lim_{n \rightarrow \infty} \frac{1}{n}\sum_{k=0}^{n-1} \langle \sigma(g^k)z, z \rangle \\
		&= \lim_{n\to\infty} \frac{1}{n} \sum_{k=0}^{n-1}\phi_z(g^k) =
		\lim_{n \rightarrow \infty} \frac{1}{n} \phi_z(g^0)%\norm{z}^2 
		= 0.
	\end{align*}
	This proves the claim.
	
	Motivated by the result of the above claim, define:
	\begin{align*}
		\mathscr{P} &:= \{g \in F: \{g^n\}_{n \in \N^*} \subset F \backslash H\}, \\
		P &:= \bigcup_{g \in \mathscr{P}} \supp(g)^c = \biggl (\,\bigcap_{g \in \mathscr{P}} \supp(g) \biggr )^c,\\
		Q &:=  P^c = \bigcap_{g \in \mathscr{P}}\supp(g).
	\end{align*}
	
	In particular, we have that $\fK\subset \bigcap_{g\in\mathscr P} \Ran(\rho_{\supp(g)}).$
	
	\textbf{Claim 2:} The set $Q$ is non-empty. Further, if $g\in F$ satisfies $Q\not\subset \supp(g)$, then there exists $n \geq 1$ for which $g^n \in H$.
	
	Using the above notation and from Claim $1$, $\rho_P(z) = 0$. If $P$ is the whole Cantor space, then $z = 0$, a contradiction. Hence, $P$ must be a proper subset and subsequently $Q$ must be non-empty. 
	Now suppose $g$ is an element in $F$ such that $Q\not\subset \supp(g)$. Then it follows $\supp(g)^c \nsubset P$ and thus $g \notin \mathscr{P}$. Hence, by definition of $\mathscr{P}$, there exists $n \geq 1$ such that $g^n \in H$ which proves the claim.
	
	The results from Claim $2$ provides a method for constructing elements in $H$ which we shall exploit in the following claims in the proof.  
	
	Define the following trees:
	\begin{center}
		\begin{tikzpicture}[baseline=0cm, scale = .8]
			\draw (0,0)--(-.5, -.5);
			\draw (0,0)--(.5, -.5);
			\draw (.5, -.5)--(0, -1);
			\draw (.5, -.5)--(1, -1);
			
			\node at (-1.3, -.5) {$L_1 = $};
			\node at (1.1, -.5) {,};		
		\end{tikzpicture}%
		\hspace*{.5em}%
		\begin{tikzpicture}[baseline=0cm, scale = .8]
			\draw (0,0)--(-.5, -.5);
			\draw (0,0)--(.5, -.5);
			\draw (-.5, -.5)--(-1, -1);
			\draw (-.5, -.5)--(0, -1);
			
			\node at (-1.8, -.5) {$R_1 = $};
		\end{tikzpicture}%
	\end{center}
	
	Then define the set of trees $\{L_i\}_{i \in \N^*}$ such that $L_{i} := f_{\tar(L_{i-1})}L_{i-1} = f_{i+1}L_{i-1}$ for $i \geq 2$ (recall $f_k$ denotes an elementary forest with a single caret at the $k$th root) and define another set of trees $\{R_i\}_{i \in \N^*}$ such that $R_{i} := f_1R_{i-1}$ for $i \geq 2$. 
	These trees are sometime called {\it vines}.
	Note that all the leaves of $L_i$ (resp.~$R_i$) are incident with left (resp.~right) edges except for the last (resp.~first) leaf. Moreover, the length of the $j$th leaf is an increasing function for $L_i$ and a decreasing function for $R_i$. This will play a key role in Claim 6.
	Further, the following identities can be easily verified:
	\[(\vert \otimes \vert \otimes L_{k-1})L_1 = L_{k+1} = f_{k+2}L_k, \quad (R_{k-1} \otimes \vert \otimes \vert)R_1 = R_{k+1} = f_1R_k\]
	for $k \geq 2$. As an example, the trees $L_2, R_2, L_3$ and $R_3$ are shown below:
	
	\begin{center}
		\begin{tikzpicture}[baseline=0cm, scale = .8]
			\draw (0,0)--(-.5, -.5);
			\draw (0,0)--(.5, -.5);
			\draw (.5, -.5)--(0, -1);
			\draw (.5, -.5)--(1, -1);
			\draw (1, -1)--(.5, -1.5);
			\draw (1, -1)--(1.5, -1.5);
			
			\node at (-1.3, -.8) {$L_2 = $};
			\node at (1.7, -.8) {,};		
		\end{tikzpicture}%
		\hspace*{.5em}%
		\begin{tikzpicture}[baseline=0cm, scale = .8]
			\draw (0,0)--(-.5, -.5);
			\draw (0,0)--(.5, -.5);
			\draw (-.5, -.5)--(-1, -1);
			\draw (-.5, -.5)--(0, -1);
			\draw (-1, -1)--(-1.5, -1.5);
			\draw (-1, -1)--(-.5, -1.5);
			
			\node at (-2.3, -.8) {$R_2 = $};
			\node at (.7, -.8) {,};		
		\end{tikzpicture}%
		\hspace*{.5em}%
		\begin{tikzpicture}[baseline=0cm, scale = .8]
			\draw (0,0)--(-.5, -.5);
			\draw (0,0)--(.5, -.5);
			\draw (.5, -.5)--(0, -1);
			\draw (.5, -.5)--(1, -1);
			\draw (1, -1)--(.5, -1.5);
			\draw (1, -1)--(1.5, -1.5);
			\draw (1.5, -1.5)--(1, -2);
			\draw (1.5, -1.5)--(2, -2);
			
			\node at (-1.3, -.8) {$L_3 = $};
			\node at (2.2, -.8) {,};		
		\end{tikzpicture}%
		\hspace*{.5em}%
		\begin{tikzpicture}[baseline=0cm, scale = .8]
			\draw (0,0)--(-.5, -.5);
			\draw (0,0)--(.5, -.5);
			\draw (-.5, -.5)--(-1, -1);
			\draw (-.5, -.5)--(0, -1);
			\draw (-1, -1)--(-1.5, -1.5);
			\draw (-1, -1)--(-.5, -1.5);
			\draw (-1.5, -1.5)--(-2, -2);
			\draw (-1.5, -1.5)--(-1, -2);
			
			\node at (-2.8, -.8) {$R_3 = $};
			\node at (.7, -.8) {.};		
		\end{tikzpicture}%
	\end{center}

	\textbf{Claim 3:} We have $[L_i, R_i]^n = [L_{in}, R_{in}]$ for all $i, n \in \N^*$ where $[L_i,R_i]^n$ is the $n$th power of the element of $F$ associated to the tree-diagram $(L_i,R_i).$
	
	We shall first prove the claim for the case when $i=1$ by induction on $n$. The initial cases for $n=1, 2$ are easily verifiable. Now suppose that $[L_1, R_1]^k = [L_k, R_k]$ for some $k \geq 2$. Then: 
	\begin{align*}
		[L_1, R_1]^{k+1} &= [L_k, R_k]\cdot [L_1, R_1] \\
		&= [f_{k+2}L_k, f_{k+2}R_k] \cdot [(R_{k-1} \otimes \vert \otimes \vert)L_1, (R_{k-1} \otimes \vert \otimes \vert)R_1] \\
		&= [L_{k+1}, f_{k+2}R_k] \cdot [f_{k+2}R_k, R_{k+1}] \\
		&= [L_{k+1}, R_{k+1}].
	\end{align*}
	This proves the statement for $i = 1$. The general case then easily follows:
	\[[L_i, R_i]^n = ([L_1,R_1]^i)^n =  [L_1, R_1]^{in} = [L_{in}, R_{in}].\] 
	
	{\bf Notation.} From now on we fix a binary sequence $u\in Q$. Such a sequence exists from Claim $2$.
	
	Recall that $t_n$ denotes the regular tree with $2^n$ leaves all with length $n$. 
	Fix a $n \in \N^*$ and denote $\nu_n$ to be the leaf of $t_n$ which corresponds to the sdi in $\cC$ that contains the sequence $u$. 
	Define the tree $p_{n,i}$ (resp.~$q_{n, i}$) by attaching copies of $L_i$ (resp.~$R_i$) to all leaves of $t_n$ except to $\nu_n$.
	Finally, define $g_{n, i} := [p_{n, i}, q_{n, i}] \in F$. 
	The purpose of the above construction is to ensure $u$ is not in the support of $g_{n,i}$ and thus $g_{n,i}$ acts like the identity on an open set containing $u$. Equivalently, 
	$$g_{n,i}\in \widehat{F_u}:=\{g\in F:\ g(u)=u, g'(u)=1\}.$$ 
	
	\textbf{Claim 4:} We have $(g_{n, i})^j = g_{n, ij}$ for all $n,i, j \in \N^*$.
	
	Fix $n \in \N^*$. Then the proof proceeds similarly as to the proof for Claim $3$ by applying the same argument separately for each of the leaves of $t_n$ except for the leaf $\nu_n$.

	\textbf{Claim 5:} There exists a \textit{subset} (not necessarily a subgroup) $H' \subset H$ such that $H' = \{g_{n, i_{n,k}} : n, k \in \N^*\}$ where for each fixed $n$ we have that the family $(i_{n, k} : k \geq 1)$ forms a strictly increasing sequence.
	
	From Claim $4$, it follows $\supp(g_{n, i}) = \supp(g_{n, j})$ for all $i, j \in \N^*$. Then by construction $u \in \supp(g_{n, i})^c$. 
	Thus $Q\not\subset\supp(g_{n,i})$ for all $n, i \in \N^*$. Now fix $n \in \N^*$ and consider $g_{n, 1} \in F$. By Claim $2$ there exists $i_1 \in \N^*$ such that $(g_{n, 1})^{i_1} = g_{n, i_1} \in H$. Set $i_{n,1} := i_1$. Similarly consider $g_{n, i_1+1}$. Again by Claim $2$ there exists $i_2 \in \N^*$ such that $(g_{n, i_1+1})^{i_2} = g_{n, (i_1+1)i_2} \in H$ and set $i_{n, 2} := (i_1+1)i_2$. Therefore, by iteratively applying this process, we obtain the set $H_n' := \{g_{n, i_{n,k}} : k \in \N^*\} \subset H$ where $(i_{n, k} : k \geq 1)$ forms a strictly increasing sequence. Repeating the above for all $n \in \N^*$ and taking $H' := \cup_{n \in \N^*}H_n' \subset H$ gives the required set which proves the claim. 
	
	\textbf{(2)} From the above claims, we have shown that $H$ contains infinitely many elements in the form $g_{n,i}$. 
	In fact, given any pair $(n,j)$ of non-zero natural numbers, there exists $i\geq j$ satisfying that $g_{n,i} \in H$. 
	Using just these elements we shall prove that the restriction of $\sigma$ on $H$ is weakly mixing. First we shall require the following claim.
	
	\textbf{Claim 6:} 
	We have that $$\lim_{i\to\infty} \langle \sigma( [L_i,R_i])\xi,\xi\rangle=0 \text{ for all } \xi\in \fH\subset \scrH.$$
	Fix $i \geq 1$ and consider $[L_i, R_i] \in F$. We identify $\xi\in\fH$ with its image $(\xi,e)$ inside $\scrH$ where $e$ stands for the trivial tree.
	Observe that $$\langle \sigma([L_i, R_i])\xi, \xi \rangle = \langle \Phi(R_i)\xi, \Phi(L_i)\xi \rangle.$$
	Moreover, 
	$$\Phi(L_i)\xi = (A\xi,AB\xi, AB^2\xi,\cdots, AB^i\xi, B^{i+1}\xi)$$
	and 
	$$\Phi(R_i)\xi = (A^{i+1}\xi, BA^i\xi, BA^{i-1}\xi,\cdots, BA\xi, B\xi).$$
	These operators are isometries from $\fH$ to $\fH\ot \ell^2(\{0,1,\cdots,i+1\})$.
	Furthermore, since $(A,B)$ is a diffuse pair we have that 
	$$\lim_iB^i\xi=\lim_iA^i\xi=0 \text{ for all } \xi\in\fH.$$ 
	This second fact allows us to forget asymptotically the right-most term of $\Phi(L_i)\xi$ and the left-most term of $\Phi(R_i)\xi$ which is crucial.
	By embedding $\ell^2(\{0,1,\cdots,i+1\})$ inside $\ell^2(\Z)$ we deduce that these two sequences of operators converge in the strong operator topology to the following isometries:
	$$L_\infty:\fH\to \fH\ot\ell^2(\Z),\ \xi\mapsto \sum_{j=0}^\infty AB^j\xi\ot\delta_j$$
	and
	$$R_\infty:\fH\to \fH\ot\ell^2(\Z),\ \xi\mapsto \sum_{j=0}^\infty BA^{j}\xi\ot\delta_{-j}.$$
	Note that we put a minus sign in the definition of $R_\infty.$ This will make our formula nicer.
	Define now the usual shift operator $S$ that we tensor by the identity:
	$$\id_{\fH}\ot S:\fH\ot\ell^2(\Z)\to \fH\ot\ell^2(\Z),\ \xi\ot\delta_n\mapsto \xi\ot\delta_{n+1}.$$
	Observe that 
	\begin{align*}
		\sigma( [L_i,R_i])\xi,\xi\rangle  = & \langle \Phi(R_i)\xi, \Phi(L_i)\xi \rangle \\
		= & \langle A^{i+1}\xi - BA^{i+1}\xi,A\xi\rangle + \langle B\xi, B^{i+1}\xi - AB^{i+1}\xi\rangle\\
		& + \langle (\id_\fH\ot S^{i+1})\circ R_\infty\xi, L_\infty\xi\rangle
	\end{align*}
	for all $i\geq 2$ and $\xi\in\fH.$
	The first two terms tend to zero in $i$ by the Cauchy-Schwarz inequality and because $A^n, B^n \xrightarrow{s} 0$ as a consequence of $(A,B)$ being diffuse.
	It is well-known and easy to prove that powers of the shift operators tends to zero in the \textit{weak} operator topology.
	Therefore, so does the sequence $(\id_\fH\ot S^{i} : i\geq 0)$ and thus the inner product of above converges to zero as $i\to\infty$ for all $\xi\in\fH$.

	\textbf{Claim 7:} The restriction $\sigma\restriction_H$ of $\sigma$ to $H$ is weakly mixing. 
	
	We will use the characterisation given by Proposition \ref{prop:unitary-rep}. 
	Fix a finite subset of vectors $K = \{x_j\}_{i = j}^k\subset\scrH$ and $\epsilon > 0$. 
	By density we can assume that $K\subset \scrK$.
	Further, for a fixed $n\geq 1$ large enough we can assume that for all $1\leq j\leq k$ we have each $x_j$ can be written in the form $[t_n, \xi^{(j,n)}]$ (where $t_n$ is the complete binary tree with $2^n$ leaves all of length $n$ and $\xi^{(j,n)}$ is a vector of $\fH^{\Leaf(t_n)}$).
	Consider one element $x_j=[t_m, \xi] \in K$ where $m \geq n$ (where we have dropped the superscript for lighter notations). Let $\nu_m$ be the label of the leaf of $t_m$ whose sdi contains $u\in Q$ as defined before Claim $4$. 
	Note for a sdp $\{\nu_i\}_{i \in I}$ we have the following identity:
	\begin{equation} \label{tau composition}
		\tau_{\nu_k}\tau^*_{\nu_l} = 
		\begin{cases}
			0,  & k \neq l \\
			\id_\scrH, & k = l
		\end{cases}
	\end{equation}
	where $\id_\scrH$ is the identity operator on $\scrH$. Then by using the notation from Equation \ref{action decomposition} and applying Equation \ref{tau composition}, we have for all $i \in \N^*$:
	\begin{align*} \label{inner product}
		\vert \langle \sigma(g_{m,i})x_j, x_j \rangle \vert &= \vert \langle \sum_{\nu \in \Leaf(t_m)} \tau^*_\nu(\sigma(\ti g_{m,i,\nu})(\tau_\nu(x_j))),x_j \rangle \vert \\
		&= \vert \langle \sum_{\nu \in \Leaf(t_m)} \tau^*_\nu(\sigma(\ti g_{m,i,\nu})(\tau_\nu(x_j))), \sum_{\omega \in \Leaf(t_m)} \tau^*_\omega(\tau_\omega(x_j)) \rangle \vert \\ 
		&= \vert \sum_{\nu \in \Leaf(t_m)} \langle \tau^*_\nu(\sigma(\ti g_{m,i,\nu})(\tau_\nu(x_j))), \tau^*_\nu(\tau_\nu(x_j)) \rangle \vert \\
		&= \vert \sum_{\nu \in \Leaf(t_m)} \langle \sigma(\ti g_{m,i,\nu})(\tau_\nu(x_j)), \tau_\nu(x_j) \rangle \vert \\
		&= \vert \sum_{\nu \neq \nu_m} \langle \sigma([L_i, R_i])[e, \xi_\nu], [e, \xi_\nu] \rangle + \norm{\xi_{\nu_m}}^2\\ 		&\leq  \sum_{\nu \neq \nu_m} \vert \langle \sigma([L_i, R_i])\xi_\nu, \xi_\nu \rangle \vert + \norm{\xi_{\nu_m}}^2.
	\end{align*}
	
	We will now show that this expression tends to zero in $m$ using the property that $(A,B)$ is diffuse.
	To keep track of all the indices $i,m,j$ we now write $\xi_\nu^{(j,m)}$ for $\xi_\nu$ if the context is unclear.
	
	For each $x_j \in K$, let $m_j \in \N$ be such that $\norm{\xi_{\nu_m}^{(j,m)}}^2< \epsilon/2$ for all $m \geq m_j$.
	Such a $m_j$ exists since $(A,B)$ is diffuse.
	Indeed, $\xi_{\nu_m}^{(j,m)}$ is a coefficient of $\xi^{(j,m)}$ where $(t_m,\xi^{(j,m)})$ is a representative of $x_j$. 
	Taking larger $m'\geq m$, we obtain that the coefficients of $\xi^{(j,m')}$ are obtained from those of $\xi^{(j,m)}$ to which we apply an increasing sequence of words in $A,B$.
	Since an increasing sequence of words in $A,B$ tends in the strong operator topology to zero we deduce that the coefficients of $\xi^{(j,m)}$ tends to zero in $m$.
	Consider now $M=\max\{m_j:\ 1\leq j\leq k)$ and fix $m\geq M$.
	We have that 
	$$\vert \langle \sigma(g_{m,i})x_j, x_j \rangle \vert \leq \epsilon/2 + \sum_{\nu \neq \nu_m} \vert \langle \sigma([L_i, R_i])\xi_\nu^{(j,m)}, \xi_\nu^{(j,m)} \rangle \vert \text{ for all } 1\leq j\leq k.$$
	Since $m$ is fixed it is easy to conclude.
	Indeed, apply Claim 6 to each vector $\xi_\nu^{(j,m)}$ with $1\leq j\leq k$ and $\nu$ leaf of $t_m$ different from $\nu_m$.
	We deduce that for $i$ large enough the quantity above is smaller than $\epsilon$ proving the claim.
	
	Therefore, we have shown the action of $H$ is weakly mixing and hence there are no non-zero finite-dimensional, $H$-invariant subspaces of $\scrH$, a contradiction.
\end{proof}

\begin{remark}\label{obs:NInd-mixing}
	\begin{enumerate}
		\item If $(\sigma,\scrH)$ is a Pythagorean representation, then it is not mixing.
		Indeed, let $H\subset F$ be the subgroup of $g\in F$ fixing the interval $[0,1/2]$. Consider now the subspace $\fK\subset\scrH$ equal to the range of the isometry $\tau_0^*$. It is not hard to see that $H$ acts trivially on $\fK$. Since $H$ is infinite this implies that $\sigma$ is not mixing. 
		\item
		It is clear that mixing implies \NInd\ for torsion-free groups. 
		Indeed, consider a mixing representation $\sigma$ of a torsion-free group $G$. 
		If $H\subset G$ is a non-trivial subgroup it must be an infinite torsion-free group and then the restriction $\sigma\restriction_H$ is weakly mixing. 
		Therefore, $\sigma\restriction_H$ does not admit any finite-dimensional subrepresentation. 
		Hence, $\sigma$ does not contain any $\Ind_H^G\theta$ with $\theta:H\act\fK$ finite-dimensional.
		\item
		By definition we trivially have that a \NInd\ representation is weakly mixing.
		Therefore, we have the chain of implications:
		\begin{center}
			mixing $\Rightarrow$ \NInd\ $\Rightarrow$ weakly mixing
		\end{center}
		for representations of torsion-free groups such as $F$.
		Moreover, we have produced a huge class of \NInd\ representations of $F$ that are not mixing. 
	\end{enumerate}
\end{remark}

With little effort, we can adjust the proof of Theorem \ref{non-induced rep theorem} to obtain the following result.
Note that the condition $A^n, B^n \xrightarrow{s} 0$ of below is much weaker than having $(A,B)$ diffuse since we only need to consider two increasing sequences of words instead of all of them.

\begin{proposition} \label{pythag rep weakly mixing corollary}
	If $A^n, B^n \xrightarrow{s} 0$, then the associated representation $\sigma$ does not contain any non-zero finite-dimensional subrepresentation.
\end{proposition}

\begin{proof}
	Continue to use the same notation defined in the proof of Theorem \ref{non-induced rep theorem}. In the proof of Theorem \ref{non-induced rep theorem} we aimed to show that the restriction of $\sigma_\fU$ to a subgroup of $F$ was weakly mixing using Proposition \ref{prop:unitary-rep}. 
	Here, we want to show that the whole representation $\sigma$ of the whole group $F$ is weakly mixing. First, observe Claim $6$ in the proof of Theorem \ref{non-induced rep theorem} still continues to hold because by initial hypothesis $A^n,B^n \xrightarrow{s} 0$ which is the only assumption used in proving Claim $6$. Thus, we have 
	$$\lim_{i \to\infty} \langle \sigma([L_i, R_i])\xi, \xi \rangle = 0 \text{ for all } \xi \in \fH.$$ 
	Define the trees $\ti p_{n,i}$ (resp. $\ti q_{n,i}$) by attaching $2^n$ copies of $L_i$ (resp. $R_i$) to each of the leaves of $t_n$ and define $\ti g_{n,i} := [\ti p_{n.i}, \ti q_{n,i}] \in F$. Note $\ti p_{n,i}, \ti q_{n,i}$ are different to $p_{n,i}, q_{n,i}$ defined in the proof of Theorem \ref{non-induced rep theorem} since here we have attached $L_i,R_i$ to \textit{all} leaves of $t_n$ (since we no longer require to fix the sequence $u$). Then by performing a similar calculation as in the proof of Claim $7$ (but now using $\ti g_{n,i}$ and not separating out the terms $\xi_{\nu_m}$) we obtain that $\sigma$ is weakly mixing which proves the result.
\end{proof}

In a future article we will see that the assumptions of Theorem \ref{non-induced rep theorem} and Proposition \ref{pythag rep weakly mixing corollary} are optimal \cite{Brothier-Wijesena-2}.

\section{A class of examples}\label{sec:example}

{\bf Definition.}
We illustrate our main result with a class of representations parametrised by the real 3-sphere $S^3$.
These examples first appear in \cite[Section 6]{BJ19}.
Consider the one-dimensional Hilbert space $\fH=\C$. 
A Pythagorean pair acting on $\C$ is a pair of complex numbers $(A,B)$ satisfying $|A|^2+|B|^2=1$.
	We deduce that the set of all Pythagorean pairs acting on $\C$ is in bijection with the real 3-sphere $S^3$.
Take $(A,B)\in S^3$ and consider the associated Pythagorean representation $\sigma=\sigma_{A,B}:F\act\scrH$.

{\bf Diffuse representations.}
Observe that $(A,B)$ is not diffuse if and only if $A$ or $B$ is equal to $0$.
Hence, the non-diffuse $(A,B)$ are the one contained in the two circles 
$$C_1:=\{(a,0)\in\C^2:\ |a|=1\} \text{ and } C_2:=\{(0,b)\in\C^2:\ |b|=1\}.$$
We deduce the following corollary of Theorem \ref{non-induced rep theorem}.

\begin{corollary}\label{cor:dim-one}
	For all $(A,B)\in S^3\setminus(C_1\cup C_2)$ we have that the associated representation $\sigma_{A,B}:F\act\scrH$ is \NInd.
\end{corollary}

We will show in a future article that all these diffuse representations of $F$ are irreducible and pairwise non-isomorphic.

{\bf Non-diffuse representations.}
The remaining representations coming from the two circles $C_1$ and $C_2$ are actually reducible. 
Indeed, it is not hard to prove that if $B=0$, then $\sigma_{A,B}$ is isomorphic to the direct sum of
\begin{itemize}
	\item the one-dimensional representation
	$\chi:F\to S^1, g\mapsto A^{\log_2((g^{-1})'(0))}$; and
	\item the monomial representation $\Ind_{F_{1/2}}^F\theta$ where 
	$F_{1/2}$ is the parabolic subgroup of $F$ fixing the point $1/2$, and 
	$\theta:F_{1/2}\to S^1,\ g\mapsto A^{\log_2((g^{-1})_+'(1/2))}.$
\end{itemize}
In the definition of $\theta$ we take the right-derivative of $g^{-1}$ at the point $1/2$.
If we rather consider the action of $g$ on the Cantor space then $\theta$ corresponds in taking a slope at the infinite binary sequence $10000\dots$.
There is a similar decomposition when $A=0$ by switching $0$ and $1$ and right-derivatives with left-derivatives.

{\bf Note:} In Section 6 of \cite{BJ19} we implicitly considered the cyclic subrepresentation of $\sigma_{A,B}$ generated by any non-zero vector $\xi\in\fH$ seating inside $\scrH$. When $A\neq 0\neq B$ it does not matter because the representation $\sigma_{A,B}$ is irreducible. Although, when $B=0$ for instance, the description of the representation given in \cite{BJ19} corresponds to the one-dimensional representation of above. In particular, if $A=1,B=0$, then the cyclic subrepresentation of $\sigma_{A,B}$ generated by $\xi$ is isomorphic to the trivial representation $1_F$.

{\bf The representations of Koopman and Garncarek.} 
It is interesting to notice that the representations considered by Garncarek \cite{garncarek2012analogs} are contained in the above Pythagorean representations.
Indeed, first consider $A=B=1/\sqrt 2.$
One can prove that $\sigma$ is unitary conjugated to the Koopman representation $\kappa:F\act L^2([0,1])$ associated to the usual action $F\act [0,1]$.
Since $F\act [0,1]$ is not measure-preserving we must include the Radon-Nikodym derivative in the formula of $\kappa$ in order to have a unitary representation.
The formula is then
$$\kappa(g)f= \left( \frac{ dg_*L}{dL} \right)^{1/2} \cdot f\circ g^{-1} \text{ for } g\in F, f\in L^2([0,1])$$
where $L$ is the Lebesgues measure.
Garncarek considered the following representations obtained by twisting the Radon-Nikodym by a parameter $s\in\R$:
$$\kappa_s(g)f:=\left( \frac{ dg_*L}{dL} \right)^{1/2+is} \cdot f\circ g^{-1} = \left( \frac{ dg_*L}{dL} \right)^{is}\cdot\kappa(g)f$$ 
for $g\in F, f\in L^2([0,1]).$
By comparing this formula with the one given in Section 6.2 of \cite{BJ19} we deduce that the representations of Garncarek are exactly the Pythagorean representations obtained from the circle 
	$$C_3:=\{ (\frac{\omega}{\sqrt 2}, \frac{\omega}{\sqrt 2}):\ \omega\in S^1\}.$$
	Note that $C_3$ is the set of $(A,B)\in S^3$ satisfying $A=B$.

\begin{remark}
	A Pythagorean representation $\sigma_{A,B}$ morally consists in applying sums of $W^* U$ to vectors of $\fH$ where $W,U$ are words in $A,B$.
	When $A=B$, then there exists a unitary $u$ satisfying that $A=B=u/\sqrt 2.$
	In that case the product $W^*U$ only depends on the difference of lengths between the words $W$ and $U$.
	This permits to describe $\sigma_{A,B}$ using the Koopman representation and a cocycle valued in the unitary group of $\fH$, see Section 6.5 of \cite{BJ19}.
	When $\dim(\fH)=1$, then this cocycle is a scalar and we recover the description of Garncarek.
	
	When $A\neq B$ but still commute, then $W^*U$ depends on the number of $A$ and $B$ appearing in the words $U,W$.	The representation $\sigma_{A,B}$ then remembers more the tree-structure of elements of the Thompson group.
	It is then more complex and cannot be defined in an obvious way via the action $F\act [0,1]$ but rather via the action on the Cantor space $\{0,1\}^{\N^*}.$ 
	Finally, when $AB\neq BA$, then $W^*U$ remembers even more the word structure. 
	This produces rich and complex families of representations for which our tools are well-adapted.
	We will investigate explicit representations arising in these contexts in future works.
\end{remark}

%%%%%%%%%%%%%%%%%%%%%%%%%%%%%%%%%%%%%%%%%%%%%%%
%%%%%%%%%%%%%%%BIBLIOGRAPHY%%%%%%%%%%%%%%%%%%%%%%

\newcommand{\etalchar}[1]{$^{#1}$}

\end{document}